\documentclass[reqno,centertags,a4paper,12pt]{amsart}
\usepackage{amssymb}
\usepackage{amsthm}
\usepackage{eucal}
\usepackage{color, soul, xcolor}
\soulregister\cite{7}
\soulregister\ref{7}
\soulregister\pageref7
\usepackage[colorlinks=true,linkcolor=blue]{hyperref}
\hypersetup{colorlinks,linkcolor={blue},citecolor={red},urlcolor={blue}}  
\usepackage[english]{babel}
\usepackage{graphicx}
\usepackage{subfig}

\newcommand{\R}{\mathbb R}
\newcommand{\s}{{\mathcal S}}
\newcommand{\Z}{\mathbb Z}
\newcommand{\T}{\mathbb T}
\newcommand{\N}{\mathbb N}

\numberwithin{equation}{section}
\newtheorem{theorem}{Theorem}[section]
\newtheorem{proposition}[theorem]{Proposition}
\newtheorem{remark}[theorem]{Remark}
\newtheorem{lemma}[theorem]{Lemma}

\newtheorem*{TA}{Theorem A}

\begin{document}
\title[Well-posedeness]{Well-posedness for a higher order water wave model on modulation spaces}

\author{X. Carvajal}
\address{Instituto de Matem\'atica, UFRJ, 21941-909, Rio de Janeiro, RJ, Brazil}
\email{carvajal@im.ufrj.br}
\author{M. Panthee}
\address{Department of Mathematics, IMECC-UNICAMP\\
13083-859, Campinas, S\~ao Paulo, SP,  Brazil}
\email{mpanthee@unicamp.br}


\begin{abstract}
 Considered in this work is the initial value problem (IVP) associated  to  a higher order water wave model
 \begin{equation*}
 \begin{cases}
\eta_t+\eta_x-\gamma_1 \eta_{xxt}+\gamma_2\eta_{xxx}+\delta_1 \eta_{xxxxt}+\delta_2\eta_{xxxxx}+\frac{3}{2}\eta \eta_x+\gamma (\eta^2)_{xxx}-\frac{7}{48}(\eta_x^2)_x-\frac{1}{8}(\eta^3)_x=0,\\
\eta(x,0) = \eta_0(x).
\end{cases}
\end{equation*}
The main interest is in addressing the well-posedness issues of the IVP when the given initial data are considered in the modulation  space $M_s^{2,p}(\R)$ or 
the $L^p$-based Sobolev spaces $H^{s,p}(\R)$, $1\leq p<\infty$. We derive some multilinear estimates in these spaces and prove that the above IVP is locally well-posed for data in $M_s^{2,p}(\R)$ whenever $s>1$ and $p\geq 1$, and in $H^{s,p}(\R)$ whenever $p\in [1,\infty)$ and $s\geq \max\left\{ \frac1{p}+\frac12, 1 \right\}$.  We also use a combination of high-low frequency technique and an {\em a priori estimate}, and prove that the local solution with data  in the modulation spaces $M_s^{2,p}(\R)$  can be extended globally to the time interval $[0, T]$ for any given $T\gg1$ if $1\leq \frac32-\frac1p <s<2$ or if $(s,p)\in [2, \infty]\times  [2, \infty]$.

\end{abstract}

\maketitle

\section{Introduction}

In this work we are interested in studying the well-posedness issues for the following  initial value problem (IVP) associated  to  a higher order water wave model
\begin{equation}\label{5kdvbbm}
\begin{cases}
\eta_t+\eta_x-\gamma_1 \eta_{xxt}+\gamma_2\eta_{xxx}+\delta_1 \eta_{xxxxt}+\delta_2\eta_{xxxxx}+\frac{3}{2}\eta \eta_x+\gamma (\eta^2)_{xxx}-\frac{7}{48}(\eta_x^2)_x-\frac{1}{8}(\eta^3)_x=0,   \\
\eta(x,0)=\eta_0(x),
\end{cases} 
\end{equation}
with given data in some function spaces which are defined via scales other than the usual $L^2$-based Sobolev spaces, viz.,  the modulation  space $M_s^{2,p}(\R)$ or 
the $L^p$-based Sobolev spaces $H^{s,p}$, $1\leq p<\infty$. 

 This higher order water wave model which possesses structure of the fifth order Korteweg-de Vries (KdV) equation  and  the Benjamin Bona Mahony (BBM) equation was proposed in \cite{BCPS1} to describe the unidirectional propagation of water waves. The authors in \cite{BCPS1} obtained this model, also known in the literature as the fifth-order KdV-BBM equation, using the second order Taylor approximation in the two-way model, the so-called $abcd-$system, introduced in \cite{BCS1, BCS2}. The coefficient in \eqref{5kdvbbm} are not arbitrary, which are  given by the following relations
\begin{equation*}
\gamma_1=\frac{1}{2}(b+d-\rho),\qquad 
\gamma_2=\frac{1}{2}(a+c+\rho), 
\end{equation*}
with $\rho = b+d-\frac16$, and 
\begin{equation*}
\begin{cases}
\delta_1=\frac{1}{4}\,[2(b_1+d_1)-(b-d+\rho)(\frac{1}{6}-a-d)-d(c-a+\rho)], \\
\delta_2 =\frac{1}{4}\,[2(a_1+c_1)-(c-a+\rho)(\frac{1}{6}-a)+\frac{1}{3}\rho], 
\gamma =\frac{1}{24}[5-9(b+d)+9\rho],
\end{cases}
\end{equation*}
where $a, b, c, d, a_1, b_1, c_1$ and $d_1$ are modelling parameter satisfying $a+b+c+d=\frac{1}{3}$, $\gamma_1+\gamma_2=\frac{1}{6}$, $\gamma=\frac{1}{24}(5-18\gamma_1)$ and $\delta_2-\delta_1=\frac{19}{360}-\frac{1}{6}\gamma_1$ with $\delta_1>0$ and $\gamma_1>0$. It is interesting to note that the higher order wather wave model \eqref{5kdvbbm}   possesses hamiltonian structure  when the coefficient $\gamma = \frac7{48}$. In this case, the following quantity 
\begin{equation}\label{energy}
 E(\eta(\cdot,t)):=  \frac12\int_{\mathbb{R}} \eta^2 + \gamma_1 (\eta_x)^2+\delta_1(\eta_{xx})^2\, dx= E(\eta_0),
\end{equation}  
remains invariant in time.

 The well-posedness issues for the IVP \eqref{5kdvbbm}  with initial data in the classical $L^2$-based Sobolev spaces $H^s(\R)$ have attracted attention of several mathematicians in recent time. While introducing this model the  authors in \cite{BCPS1} proved that the IVP \eqref{5kdvbbm} is locally well-posed for initial data in $H^s(\R)$ when  $s\geq 1$. The authors in \cite{BCPS1} also proved that the local solution can be extended globally in time when the initial data possesses Sobolev regularity $s\geq \frac32$. More precisely,  they used the energy conservation law \eqref{energy}  to obtained an {\em a priori} estimate in $H^2(\R)$  to prove the local solution can be extended globally in time for  data in $H^s(\R)$, $s\geq 2$. While, for data with  regularity $ \frac32\leq s< 2$, {\em splitting to high-low frequency parts} technique introduced in \cite{B1, B} (see also \cite{BS}) was used.  Further, this technique of  splitting initial data to high-low frequency parts was used in \cite{CP} more efficiently to prove that the IVP \eqref{5kdvbbm} is globally well-posed   for initial data with Sobolev regularity $s\geq 1$. More precisely, they proved the following result.
 
\begin{TA}  Assume $\gamma_1, \delta_1 > 0$.  
 Let $1\leq s< 2$ and $\gamma=\frac7{48}$.  Then  for any given $T>0$, the local solution to the IVP \eqref{5kdvbbm}   can be extended to the  time interval $[0, T]$.  Hence the IVP \eqref{5kdvbbm} is globally well-posed in this case. In addition if $\eta_0 \in H^s$, one also has
\begin{equation*}
\eta(t) -S(t)\eta_0 \in H^2, \quad \textrm{for all time}\; t\in [0, T]
\end{equation*}
and
\begin{equation}\label{Norm2}
\sup_{t \in [0,T]}\|\eta(t) -S(t)\eta_0\|_{H^2} \lesssim (1+T)^{2-s},
\end{equation}
where $S(t)$ is the unitary group associated to the linear problem as defined in \eqref{St} below.
\end{TA} 
  Furthermore, the authors in \cite{CP} also proved that the well-posedness result obtained in  \cite{BCPS1}  is sharp by showing  that the application that takes initial data  in $H^s(\R)$ to the solution fails to be continuous at the origin for $s<1$. The IVP \eqref{5kdvbbm}  posed on domains other than $\R$ is also considered in the literature. We refer \cite{CPP-1} for the similar well-posedness results of the IVP \eqref{5kdvbbm} when posed on the periodic domain and \cite{hongqiu} when posed on quarter plane. 
 
 Study of other qualitative properties of the solution to the IVP \eqref{5kdvbbm} has also attracted attention in the literature. Recently, the authors in \cite{CP-22} proved that the regularity in the initial data of the IVP \eqref{5kdvbbm} propagates in the solution; in
other words, no singularities can appear or disappear in the solution to this model. They also considered the IVP \eqref{5kdvbbm} in the space of the analytic functions, the so-called Gevrey class and proved  the local well-posedness results for data in such spaces.   They also studied  the evolution of radius of analyticity in such class by providing explicit formulas for upper and lower bounds. Quite recently, the authors in \cite{BTT-23} obtained a better lower bound for the evolution of radius of analyticity of the solution  to the IVP \eqref{5kdvbbm}. More precisely, the authors in \cite{BTT-23} showed that the uniform radius of spatial analyticity $\sigma(t)$ of the solution to the IVP \eqref{5kdvbbm} at time $t$ cannot decay
faster than $\frac{c}t$, $c>0$, for large $t>0$ for given initial data that is analytic with fixed
radius $\sigma_0$. Finally, for further theory on model in \eqref{5kdvbbm} we refer to \cite{BCG}.

As discussed in the previous paragraph, so far in the existing literature, the well-posedness issues of the IVP \eqref{5kdvbbm} are explored considering the initial data  in the classical $L^2$-based Sobolev spaces $H^s$ of order $s\in \R$ with norm
$$
\|f\|_{H^s}^2 = \int_{\R}\langle\xi\rangle^{2s}|\widehat f(\xi)|^2\,d\xi,
$$
where  $\langle\cdot\rangle = 1+|\cdot|$ and $\widehat f$ denotes the Fourier transform of $f$ defined by
\begin{equation*}
\widehat{f}(\xi) =
\frac{1}{\sqrt{2\pi}}\int_{\mathbb{R}}e^{-ix\xi}f(x)\,dx.
\end{equation*}

In recent time, well-posedness issues of the IVPs associated to the nonlinear dispersive equations have been studied in some other scales of the function spaces than the usual $L^2$ based Sobolev spaces $H^s(\R)$. Most commonly used such spaces are the Fourier-Lebesgue spaces $\mathcal{F}L^{s,p}(\R)$ with norm 
 $$
 \|u\|_{\mathcal{F}L^{s,p}} =\|\langle \xi\rangle^s \widehat{u}(\xi)\|_{L^p},
 $$
  modulation spaces  $M_s^{r,p}(\R)$  with norm given by \eqref{def-2} (see Section \ref{sec-2} below) and the $L^p$-based Sobolev spaces of order $s$,  $H^{s,p}$, with norm
\begin{equation}\label{HspNorm}
\|f\|_{H^{s,p}}^p = \int_{\R}|(\Lambda^s f)(\xi)|^p\,d\xi,
\end{equation}
where the $\Lambda^s$ is the multiplier operator with symbol $(1+|\xi|^2)^{s/2}$ i.e.
$$
\widehat{\Lambda^s f}(\xi)=(1+|\xi|^2)^{s/2} \widehat{f}(\xi).
$$
By Plancherel's identity we have $\|f\|_{H^{s,2}}\sim \|f\|_{H^s}$,  and when
 $r=2$ $$\mathcal{F}L^{s,p}(\R)\subset M_s^{2,p}(\R).$$
  
  Taking in consideration the discussion above, our main interest in this work is to investigate the well-posedness issues for the IVP \eqref{5kdvbbm} with given data in the modulation spaces  $M_s^{r,p}(\R)$ and $L^p$-based Sobolev spaces $H^{s,p}$ of order $s$ considering appropriate values of $r$ and $p$.

  As a motivation, in what follows, we mention some results in this direction from the existing literature. We start by  mentioning  the local well-posedness result obtained in \cite{Gr-04} for the modified Korteweg-de Vries (mKdV) equation for given data in $\mathcal{F}L^{s,p}(\R)$ for $s\geq \frac1{2p}$ with $2\leq p<4$  and its improvement proved in \cite{GV-09} for the same range of $s$ with $2\leq p <\infty$. Recently, the authors in \cite{Oh-Wang}  investigated the IVP associated to the mKdV equation in the complex case and proved the local well-posedness result for given data in the modulation spaces  $M_s^{2,p}(\R)$ whenever $s\geq \frac14$ and $2\leq p <\infty$. Quite recently, the same authors in \cite{OW-20} proved that the cubic nonlinear Schr\"odinger (NLS) equation is globally well-posed in $M^{2,p}(\R)$ for any $1\leq p<\infty$ and the normalized cubic NLS is globally well-posed in $\mathcal{F}L^p(\T)$ for any $1\leq p<\infty$ by introducing a new function space $H\!M^{\theta,p}$ whose norm is given by the $\ell^p$-sum of the modulated $H^{\theta}$-norm of a given function that agrees with the modulation space $M^{2,p}(\R)$ on the real line and Fourier-Lebesgue space $\mathcal{F}L^p(\T)$ on the circle.   Also, we mention a very recent work of the authors in \cite{CP-22a}, where the extended nonlinear Schr\"odinger (e-NLS) equation and the higher order nonlinear Schr\"odinger (h-NLS)  equation were considered and proved to be locally well-posed for given data in  the modulation space $M_s^{2,p}(\R)$ respectively for $s>-\frac14$ and $s\geq\frac14$ with $2\leq p<\infty$. Further, recent works \cite{BV-21} and \cite{BV-20} that respectively deal with  the BBM equation and Boussinesq equation in the modulation spaces are worth mentioning.


Now we present the structure of this work. In Section
 \ref{sec-2},  we introduce the function spaces, their properties  and state the main results of this work. Section \ref{sec-3} is devoted to derive multilinear estimates in modulation spaces and prove the well-posedness results for data in this space. We prove global well-posedness result for such data in Section \ref{sec-4}.  Section \ref{sec-5} is devoted to provide the local well-posedness result for data in the $L^p$-bases Sobolev  spaces. Finally, in Section \ref{sec-6} we report the existence of solitary wave solution in the explicit form under certain condition on the parameters appearing in the model.  We finish this section recording some principal notations that will be used throughout this work.\\

\noindent
{\textbf{Notations:}} We will use standard notations of the PDEs 
throughout this work.  We use $C$ or $c$ to denote constant whose value may differ from one line to the next. We write $A\lesssim B$
if there exists a constant $c >0$ such that $ A \leq cB$, we also write $A \sim B$ if $A\lesssim B$ and $B\lesssim A$... etc.

\section{Function spaces and statement of the main results}\label{sec-2} 

In this section we introduce the function spaces on which we are interested to work in and the main results of this work.

\subsection{Function Spaces}

We already described the $L^p$-based Sobolev space  $H^{s,p}(\R)$, $s\in\R$, $p\in[1, \infty)$, in the Introduction with the norm given by \eqref{HspNorm}. In what follows, we describe the modulation spaces first introduced in 1983 by H. G. Feichtinger \cite{Feich-83} as part of an attempt to define smoothness spaces over locally compact Abelian groups. It is worth noticing that  H. G. Feichtinger  also proposed a theory that parallels that of the better known Besov spaces, and the modulation spaces have been  a motivation for the more general theory of coorbit spaces \cite{Feich-1, Feich-2, Feich-88}. An expository paper by the same author  \cite{F-06}  provides a historical perspective on what led to the invention of modulation spaces. In this work, we will use modulation space defined via equivalent norm as described in \cite{Oh-Wang}. Let  $\sigma\in \mathcal{S}(\R)$ be such that
$${\mathrm{supp}}\, \sigma\subset[-1, 1], \qquad \sum_{k\in\Z}\sigma(\xi-k)=1,$$
and, for given $n\in \Z$, $\Pi_n$ be a Fourier multiplier operator with symbol given by
$$\sigma_n(\xi):=\sigma(\xi-n),$$
i.e., 
\begin{equation}\label{Pn}
\widehat{\Pi_n f}(\xi):=\sigma_n(\xi)\widehat{f}(\xi).
\end{equation}
For given $s\in \R$ and  $1\leq r, p\leq\infty$,  we define modulation spaces 
    $M_s^{r,p}(\R)$ as follows  
    \begin{equation}
    \label{def-1}
      M_s^{r,p}(\R):=\{f\in \mathcal{S}'(\R) : \|f\|_{M_s^{r,p}}<\infty\},
      \end{equation}
where
\begin{equation}
    \label{def-2}\|f\|_{M_s^{r,p}}:= \|\langle n\rangle^s\|\Pi_nf\|_{L_x^r(\R)}\|_{\ell_n^p(\Z)},
\end{equation}

For given dyadic number $N\geq 1$, let $P_N$ be the Littlewood-Paley projector on the frequencies $\{|\xi|\sim N\}$. 
The following estimates, for any $1\leq q\leq p\leq \infty$,  follow readily from Bernstein's inequality
\begin{equation}
\label{bern-1}
\begin{split}
&\|P_N f\|_{L_x^p}\lesssim N^{\frac1q-\frac1p}\|f\|_{L_x^q},\\
&\|\Pi_n f\|_{L_x^p}\lesssim\|f\|_{L_x^q}.
\end{split}
\end{equation}

If we define  $u_n :=\Pi_n u$, the norm of the modulation space $M_s^{r,p}$ defined in \eqref{def-2} for $r=2$ can be written in the form
\begin{equation*}
\|f\|_{M_s^{2,p}}:= \|\|u_n\|_{H^s(\R)}\|_{\ell_n^p(\Z)}. 
\end{equation*}

In order to make the exposition more simple, in the case when $r=2$, we introduce a notation  $M^{s,p}:=M^{2, p}_s(\R)$.
It is worth noticing  that, for $p\geq 2$ one has $\ell^2\subseteq \ell^p$ and consequently $H^s\subseteq M^{s, p}$ i.e. 
\begin{equation}\label{immersion1}
\|f\|_{M^{s,p}}\lesssim \|f\|_{H^s(\R)}
\end{equation}

The following elementary results on the sum of dyadic numbers less than or equal to a given dyadic number $M$ will also be useful in our analysis.

\begin{lemma}\label{lem32} Let $a\neq 0$ and $M\geq 2$ be a given dyadic number, then
\begin{equation}\label{dyd-1}
\sum_{1\leq N\lesssim M} N^a\quad
\begin{cases}
 \sim_a M^a \quad & \textrm{if} \quad a>0,\\
 \lesssim 1\quad &\textrm{if} \quad a<0,
 \end{cases}
\end{equation}
and 
\begin{equation}\label{dyd-2}
\quad \sum_{1\leq N\lesssim M} 1 \lesssim \ln M.
\end{equation}
\end{lemma} 
\begin{proof}
Let $M=2^p$ and $a>0$, then 
$$
\sum_{1\leq N\leq C M} N^a=\sum_{1\leq 2^j \leq 2^b 2^p}(2^j )^a=\sum_{0\leq j \leq p+b}2^{ja}= \dfrac{2^{a(p+b)}-1}{2^a-1}\lesssim_a M^aC^a\sim_a M^a.
$$
where $C=2^b$. If $a<0$ the infinite geometric series is convergent and this finishes the proof of \eqref{dyd-1}. 

Similarly
$$
\sum_{1\leq N\leq C M} 1=\sum_{1\leq 2^j \leq 2^{b+p}} 1=\sum_{0\leq j \leq b+p}1 =p+b+1=\dfrac{\ln M}{\ln 2}+ \dfrac{\ln C}{\ln 2}+1\lesssim \ln M.
$$
where we used the hypothesis that $M\geq 2$.
\end{proof}

\subsection{Statement of the main results}

In this subsection, we state the  main result of this work. We start with the  well-posedness results for the IVP \eqref{5kdvbbm} with given data in the modulation space $M^{s,p}(\R)$. The first main result is the following local well-posedness theorem.
\begin{theorem}\label{mainTh1}
Let $s> 1$ and $p\geq 1$, then the IVP \eqref{5kdvbbm} is locally well-posed in $M^{s, p}(\R)$.
\end{theorem}

We will use the splitting argument introduced in \cite{B1, B} and way earlier in \cite{BS} to prove that the local solution of the IVP \eqref{5kdvbbm} for given data in the modulation space can be extended to any arbitrary large time interval. More precisely, we prove the following global well-posedness result.

\begin{theorem}\label{Th-global}
Let $T\gg 1$ be given,  and $\gamma=\frac7{48}$. Then the local solution of the  IVP \eqref{5kdvbbm} with initial data $\eta_0\in M^{s, p}(\R)$ given by Theorem \ref{mainTh1} can be extended to the time interval $[0, T]$  if $1\leq \frac32-\frac1p <s<2$ or if $(s,p)\in [2, \infty]\times  [2, \infty]$.
\end{theorem}

The next main result deals with the local well-posedness of the  IVP \eqref{5kdvbbm} for given data in the $L^p$-based sobolev space $H^{s,p}$ and reads as follows.

\begin{theorem} \label{mainTh2}
Let $p\in [1,\infty)$ and $s\geq \max\left\{ \frac1{p}+\frac12, 1 \right\}$, then the fifth order KdV-BBM type equation \eqref{5kdvbbm} with $\delta_2=0$ is locally well-posed in $H^{s,p}(\R)$. If $\delta_2 \neq 0$ is locally well-posed in $H^{s,2}(\R)$
\end{theorem}



\setcounter{equation}{0}
\section{Local Well-posedness Theory in $M^{s, p}$, $s> 1$, $p\geq1$.}\label{sec-3}



This section is devoted to investigate the local well-posedness issues of the IVP \eqref{5kdvbbm}
for given data  $\eta(x,0) =\eta_0(x)$ in the modulation space $M^{s,p}(\R)$. We start by writing the IVP (\ref{5kdvbbm}) in an equivalent 
integral equation form. 
First note that, taking the Fourier transform in  equation \eqref{5kdvbbm} with respect to the spatial variable we obtain
\begin{equation}\label{eq1.9}
\begin{cases}
i\eta_t = \phi(\partial_x)\eta + \tau (\partial_x)\eta^2 - \frac18\psi(\partial_x)\eta^3  -\frac7{48}\psi(\partial_x)\eta_x^2\, ,\\
 \eta(x,0) = \eta_0(x),
 \end{cases}
\end{equation}
where  the  Fourier multiplier operators $\phi(\partial_x)$, $\psi(\partial_x)$ and $\tau(\partial_x)$ are defined via
their symbols, {\it viz.},
\begin{equation}\label{phi-D}
\widehat{\phi(\partial_x)f}(\xi):=\phi(\xi)\widehat{f}(\xi), \qquad \widehat{\psi(\partial_x)f}(\xi):=\psi(\xi)\widehat{f}(\xi) \;\;\; {\rm and }\;\;\; \widehat{\tau(\partial_x)f}(\xi):=\tau(\xi)\widehat{f}(\xi), 
\end{equation}
 with
\begin{equation*}
  \phi(\xi)=\frac{\xi(1-\gamma_2\xi^2+\delta_2\xi^4)}{\varphi(\xi)}, \quad \psi(\xi)=\frac{\xi}{\varphi(\xi)} \quad  {\rm and} \quad \tau(\xi)=\frac{3\xi-4\gamma\xi^3}{4\varphi(\xi)}.
\end{equation*}

Since $\gamma_1,  \delta_1$ are positive, the common denominator 
\begin{equation*}
 \varphi(\xi) := 1 + \gamma_1\xi^2+\delta_1\xi^4,
\end{equation*}
in the above symbols is strictly positive.  

From here onwards, we will consider the IVP \eqref{eq1.9} in place of the \eqref{5kdvbbm}. Now, we consider first the following  linear IVP  associated to \eqref{eq1.9}
\begin{equation*}
\begin{cases}
i\eta_t = \phi(\partial_x)\eta,\\
\eta(x,0) = \eta_0(x),
\end{cases}
\end{equation*}
whose solution is given  by $\eta(t) = S(t)\eta_0$, where 
\begin{equation}\label{St}\widehat{S(t)\eta_0} = e^{-i\phi(\xi)t}\widehat{\eta_0}
\end{equation}
 is defined via its Fourier transform.
Clearly, $S(t)$ is a unitary operator on $H^s$ and $M^{s, p}$ for any $s \in \R$, so that
\begin{equation}\label{eq1.11}
\|S(t)\eta_0\|_{H^s} = \|\eta_0\|_{H^s},\qquad \textrm{and}\qquad \|S(t)\eta_0\|_{M^{s, p}} = \|\eta_0\|_{M^{s,p}} 
\end{equation}
for all $t > 0$.

Finally, we use Duhamel's formula  to write the IVP  \eqref{eq1.9} in the following equivalent integral equation form,
\begin{equation}\label{eq1.12}
\eta(x,t) = S(t)\eta_0 -i\int_0^tS(t-t')\Big(\tau(\partial_x)\eta^2 - \frac18 \psi(\partial_x)\eta^3 -\frac7{48}\psi (\partial_x)\eta_x^2\Big)(x, t') dt'.
\end{equation}

In what follows, a short-time solution of (\ref{eq1.12}) will be obtained via the contraction mapping principle in the space $C([0,T];M^{s,p})$.  This will provide a 
proof of Theorem \ref{mainTh1}.

\subsection{Multilinear Estimates}

 In this subsection we establish several multilinear estimates that will be useful to perform contraction mapping principle in order to provide a proof of the local well-posedness result announced in  Theorem \ref{mainTh1}.
  First, we record the  following $M^{s, p}$ version of the  ``sharp" bilinear estimate obtained in \cite{BT}.

\begin{lemma}\label{BT1}
 For $s > 0$, $p\geq 1$, there is a constant $C = C_s$ for which
\begin{equation}\label{bt}
\|\omega(\partial_x) (u_1 u_2)\|_{M^{s, p}} \le C\|u_1\|_{M^{s, p}}\|u_2\|_{M^{s, p}}
\end{equation}
 where $\omega(\partial_x)$  is the Fourier multiplier operator 
with symbol
\begin{equation*}
\omega(\xi) \, = \, \frac{|\xi|}{1 + \xi^2} \leq \dfrac{2}{\langle \xi \rangle}.
\end{equation*} 
\end{lemma}

\begin{proof}
We divide the proof of this lemma in two different case. First, we consider the case $p>1$.

Note that, $\left(M^{s, p}\right)'=M^{-s, p'}$, where $\frac1p+\frac1{p'}=1$. Therefore, by  duality,  we have
\begin{equation}\label{bt1}
\|\omega(\partial_x) (u_1 u_2)\|_{M^{p, s}} =\sup_{v \in M^{-s, p'}}\int \omega(\partial_x) (u_1 u_2) (x)\overline{v}(x) dx.
\end{equation}

Thus, using Plancherel's equality,  in view of \eqref{bt1} the estimate \eqref{bt} is equivalent to proving that
\begin{equation}\label{bt2}
\int_{\xi_1+\xi_2+\xi=0}\omega(\xi)\widehat{v}(\xi)\widehat{u_1}(\xi_1) \widehat{u_2}(\xi_2) 
\lesssim \|u_1\|_{M^{s, p}}\|u_2\|_{M^{s,p}}\|v\|_{M^{-s, p'}}.
\end{equation}

Note that, the estimate \eqref{bt2} is a consequence of the following inequality
\begin{equation}\label{bt3}
\int_{\xi_1+\xi_2+\xi=0}\dfrac{\langle \xi \rangle^s}{\langle \xi \rangle}\widehat{v}(\xi)\widehat{u_1}(\xi_1) \widehat{u_2}(\xi_2) 
\lesssim \|u_1\|_{M^{s,p}}\|u_2\|_{M^{s, p}}\|v\|_{M^{0,p'}}.
\end{equation}
Let $N_1$, $N_2$ and $N$ be dyadic numbers such that
$$
|\xi_1|\sim N_1,\quad|\xi_2|\sim N_2,\quad |\xi|\sim N,
$$
and let 
$$
|\xi_{\textrm max}|= \max \{|\xi_1|, |\xi_2|\}.
$$
Without loss of generality we can suppose that $|\xi_1|\geq |\xi_2|$. As $\xi=-\xi_1-\xi_2$, one has $|\xi|\leq 2|\xi_1|$.

 In what follows, we prove the estimate \eqref{bt3} considering two different cases, viz., $N_2 \ll N_1$ and $N_2 \sim N_1$.\\

\noindent
{\bf Case I. $\boxed{N_2 \ll N_1}$:}
 In this case we have $|\xi|=|\xi_1+\xi_2| \gtrsim |\xi_1|$ and consequently
 $$|\xi|\sim |\xi_1|\sim N_1.$$
 
 Using $u_N=P_Nu$ and Plancherel's inequality, to obtain \eqref{bt3} in this case, we need to estimate
 \begin{equation}\label{bt3.1}
 \mathcal{X}:=\sum_{N_2 \ll N_1\sim N} \int \dfrac1{N^{1-s}} u_{N_1}u_{N_2} v_{N}.
 \end{equation}
 
 Using Cauchy-Schwartz and Bernstein's inequality  \eqref{bern-1}, we obtain from \eqref{bt3.1} that 
 \begin{equation}\label{bt4.0}
 \begin{split}
 \mathcal{X}\leq &\sum_{N_2 \ll N_1\sim N}\dfrac1{N^{1-s}} \|u_{N_1}\|_{L^2}\| u_{N_2}\|_{L^\infty} \|v_{N}\|_{L^2}\\
 \leq &\sum_{N_2 \ll N_1\sim N}\dfrac1{N^{1-s}} \|u_{N_1}\|_{L^2}N_2^{1/2}\| u_{N_2}\|_{L^2} \|v_{N}\|_{L^2}\\
 \lesssim &\sum_{N_2 \ll N_1\sim N}\dfrac{N_2^{1/2}}{N^{1-s}} N_1^{s}\|u_{N_1}\|_{L^2}N_2^{s}\| u_{N_2}\|_{L^2} \|v_{N}\|_{L^2}N_1^{-s}N_2^{-s}\\
 \lesssim &\sum_{N_2 \ll N_1\sim N}\dfrac{N_2^{1/2} N^{-s}N_2^{-s}}{N^{1-s}} \|u_{N_1}\|_{H^s}\| u_{N_2}\|_{H^s} \|v_{N}\|_{L^2},\\
 \end{split}
 \end{equation}
 where we used that $N\sim N_1$. Now opening the sum in \eqref{bt4.0}, using H\"older's inequality and Lemma \ref{lem32}, we have
 \begin{equation}\label{bt4.1}
 \begin{split}
 \mathcal{X}
 \lesssim &\sum_{N \gg 1} \dfrac{\|v_{N}\|_{L^2}}{N}\sum_{ N_1\lesssim N} \|u_{N_1}\|_{H^s} \Big(\sum_{ N_2\leq N}\| u_{N_2}\|_{H^s} N_2^{1/2-s}\Big)\\
  \lesssim &\sum_{N \gg 1} \dfrac{\|v_{N}\|_{L^2}}{N}\sum_{ N_1\lesssim N} \|u_{N_1}\|_{H^s} \Big(\| u_2\|_{M^{s,p}}\Big(\sum_{ N_2\leq N} N_2^{p'(1/2-s)}\Big)^{1/p'}\Big).
 \end{split}
 \end{equation}
 
 Observe that, using the estimates \eqref{dyd-1} and \eqref{dyd-2} from Lemma \ref{lem32}, we have
 \begin{equation}\label{bt5.0}
 \Big(\sum_{ N_2\lesssim N} N_2^{p'(1/2-s)}\Big)^{1/p'} \lesssim
 \begin{cases}
 1 \quad \textrm{if}\quad s>1/2,\\
 \left( \ln N\right)^{1/p'}\quad \textrm{if}\quad s=1/2,\\
 N^{1/2-s}\quad \textrm{if}\quad  s<1/2.
 \end{cases}
 \end{equation}
 This motivates us to analyse \eqref{bt4.1}  by dividing in the  following three different cases.\\
 
 \noindent
 {\bf Sub-Case Ia. $\boxed{0<s<1/2}$:}
 In this case, using  H\"older's inequality and \eqref{bt5.0}, we have from \eqref{bt4.1} that
\begin{equation*}
 \begin{split}
 \mathcal{X}
 &\lesssim \| u_2\|_{M^{s,p}} \sum_{N \geq 1} \dfrac{\|v_{N}\|_{L^2}}{N^{1/2+s}}
  \,\,\|u_1\|_{M^{s,p}} \Big(\sum_{N_1 \lesssim N}1\Big)^{1/p'}\\
   &\lesssim \| u_2\|_{M^{s,p}} \|u_1\|_{M^{s,p}}  \sum_{N \geq 1} \dfrac{\left(\ln N\right)^{1/p'}}{N^{1/2+s}}\|v_{N}\|_{L^2}
  \\
   &\lesssim \| u_2\|_{M^{s,p}} \|u_1\|_{M^{s,p}}  \|v\|_{M^{0,p'}}  \left(\sum_{N \geq 1} \dfrac{\left(\ln N\right)^{p/p'}}{N^{(1/2+s)p}}\right)^{1/p}\\
   &\lesssim \| u_2\|_{M^{s,p}} \|u_1\|_{M^{s,p}}  \|v\|_{M^{0,p'}}.
 \end{split}
 \end{equation*}
{\bf Sub-Case Ib. $\boxed{s=1/2}$:}
In this case also using  H\"older's inequality and \eqref{bt5.0}, we have from \eqref{bt4.1} that
\begin{equation*}
 \begin{split}
 \mathcal{X}
 &\lesssim \| u_2\|_{M^{1/2,p}} \, \sum_{N \gg 1} \dfrac{\Big( \ln N\Big)^{1/p'}}{N}\|v_{N}\|_{L^2}\|u_1\|_{M^{1/2,p}} \Big(\sum_{N_1 \lesssim N}1\Big)^{1/p'}\\
 &\lesssim \| u_2\|_{M^{1/2,p}} \|u_1\|_{M^{1/2,p}}  \|v\|_{M^{0,p'}} \left(\sum_{ N\gg 1} \frac{(\ln N)^{2p/p'}}{N^{p}}\right)^{1/p}\\
 &\lesssim \| u_2\|_{M^{1/2,p}} \|u_1\|_{M^{1/2,p}}  \|v\|_{M^{0,p'}}.
 \end{split}
 \end{equation*}
{\bf Sub-Case Ic. $\boxed{s>1/2}$:}
As in the earlier cases,  using \eqref{bt5.0}, H\"older's inequality and Lemma \ref{lem32}, we obtain from \eqref{bt4.1}
\begin{equation*}
 \begin{split}
 \mathcal{X}
 \lesssim &\| u_2\|_{M^{s,p}} \, \sum_{N \gg 1} \dfrac{1}{N}\|v_{N}\|_{L^2}\|u_1\|_{M^{s,p}} \Big(\sum_{N_1 \lesssim N}1\Big)^{1/p'}\\
 \lesssim &\| u_2\|_{M^{s,p}} \|u_1\|_{M^{s,p}}  \|v\|_{M^{0,p'}} \left(\sum_{ N\gg 1} \frac{(\ln N)^{p/p'}}{N^{p}}\right)^{1/p}\\
 \lesssim &\| u_2\|_{M^{s,p}} \|u_1\|_{M^{s,p}}  \|v\|_{M^{0,p'}}.
 \end{split}
 \end{equation*}

\noindent
{\bf Case II. $\boxed{N_2 \sim N_1}$:}
We analyse this case  dividing in two different sub-cases, viz., considering  $N \sim N_2 \sim N_1$ and $N \ll N_2 \sim N_1$.
\\

\noindent
{\bf Sub-Case IIa. $\boxed{N \sim N_2 \sim N_1}$:}
This Sub-case follows as in the Case I.\\

\noindent
{\bf Sub-Case IIb. $\boxed{N \ll N_2 \sim N_1}$:}
In this Sub-case, to obtain \eqref{bt3} we need to estimate the term
 \begin{equation}\label{Y}
 \mathcal{Y}:=\sum_{N \ll N_2\sim N_1} \int \dfrac1{N^{1-s}} u_{N_1}u_{N_2} v_{N}.
 \end{equation}
 
 Using Cauchy-Schwartz and Bernstein's inequality \eqref{bern-1}, we obtain from \eqref{Y} that
 \begin{equation}\label{bt4}
 \begin{split}
 \mathcal{Y}\leq &\sum_{N \ll N_2\sim N_1}\dfrac1{N^{1-s}} \|u_{N_1}\|_{L^2}\| u_{N_2}\|_{L^2} \|v_{N}\|_{L^\infty}\\
 \leq &\sum_{N \ll N_2\sim N_1}\dfrac1{N^{1-s}} \|u_{N_1}\|_{L^2}\| u_{N_2}\|_{L^2} N^{1/2}\|v_{N}\|_{L^2}\\
 \lesssim &\sum_{N \ll N_2\sim N_1}\dfrac{N^{1/2}}{N^{1-s}} N_1^{s}\|u_{N_1}\|_{L^2}N_2^{s}\| u_{N_2}\|_{L^2} \|v_{N}\|_{L^2}N_1^{-s}N_2^{-s}\\
 \lesssim &\sum_{N \ll N_2\sim N_1}\dfrac{N^{1/2} N_1^{-s}N_2^{-s}}{N^{1-s}} \|u_{N_1}\|_{H^s}\| u_{N_2}\|_{H^s} \|v_{N}\|_{L^2}.
 \end{split}
 \end{equation}
 
 Opening the sum in \eqref{bt4}, using H\"older's inequality and Lemma \ref{lem32}, we have
 \begin{equation}\label{bt4.2}
 \begin{split}
 \mathcal{Y}
 \lesssim &\sum_{N_1 \gg 1} N_1^{-s}\|u_{N_1}\|_{H^s}\sum_{ N_2\lesssim N_1} N_2^{-s}\|u_{N_2}\|_{H^s} \Big(\sum_{ N\lesssim N_1}\| v_N\|_{L^2} \dfrac{1}{N^{1/2-s}}\Big)
 \\
 \lesssim &\sum_{N_1 \gg 1} N_1^{-s}\|u_{N_1}\|_{H^s}\sum_{ N_2\lesssim N_1} N_2^{-s}\|u_{N_2}\|_{H^s}\Big(\| v\|_{M^{0,p'}}\Big(\sum_{ N\lesssim N_1} \dfrac1{N^{p(1/2-s)}}\Big)^{1/p}\Big).
 \end{split}
 \end{equation}
 
 Observe that, by Lemma \ref{lem32} we get
 \begin{equation}\label{bt5}
 \Big(\sum_{ N\lesssim N_1} \dfrac1{N^{p(1/2-s)}}\Big)^{1/p} \lesssim
 \begin{cases}
 1 \quad \textrm{if}\quad s<1/2,\\
 \left( \ln N_1\right)^{1/p}\quad \textrm{if}\quad s=1/2,\\
 N_1^{s-1/2}\quad \textrm{if}\quad  s>1/2.
 \end{cases}
 \end{equation}
 
 Taking in consideration \eqref{bt5} we analyse \eqref{bt4.2} taking the following three different cases.
 \\
 
 \noindent
 {\bf Sub-Case IIb i) $\boxed{s>1/2}$:}
 Using \eqref{bt5}, H\"older's inequality and Lemma \ref{lem32}, we have
\begin{equation*}
 \begin{split}
 \mathcal{Y}
 \lesssim &\| v\|_{M^{0,p'}} \, \sum_{N_1 \geq 1} \dfrac1{N_1^{1/2}}\|u_{N_1}\|_{H^s} \Big(\sum_{ N_2\lesssim N_1}  \dfrac1{N_2^{sp'}}\Big)^{1/p'}\|u_{2}\|_{M^{s,p}}\\
 \lesssim & \| v\|_{M^{0,p'}} \|u_{2}\|_{M^{s,p}}\|u_{1}\|_{M^{s,p}}\, \Big(\sum_{N_1 \geq 1} \dfrac1{N_1^{p'/2}}\Big)^{1/p'}\\
 \lesssim & \| v\|_{M^{0,p'}} \|u_{2}\|_{M^{s,p}}\|u_{1}\|_{M^{s,p}}.
 \end{split}
 \end{equation*}

\noindent
{\bf Sub-Case IIb ii) $\boxed{s=1/2}$:}
In this case too using \eqref{bt5} and H\"older's inequality, we  obtain from \eqref{bt4.2} that
\begin{equation*}
 \begin{split}
 \mathcal{Y}
 \lesssim &\| v\|_{M^{0,p'}} \, \sum_{N_1 \geq 1} N_1^{-1/2}\left( \ln N_1\right)^{1/p}\|u_{N_1}\|_{H^s}\Big(\sum_{ N_2\lesssim N_1} \frac1{N_2^{p'/2}}\Big)^{1/p'}\|u_{2}\|_{M^{s,p}}\\
 \lesssim &\| u_2\|_{M^{1/2,p}} \|u_1\|_{M^{1/2,p}}  \|v\|_{M^{0,p'}} \left(\sum_{ N_1\geq 1} \frac{(\ln N_1)^{p'/p}}{N_1^{p'/2}}\right)^{1/p'}\\
 \lesssim &\| u_2\|_{M^{1/2, p}} \|u_1\|_{M^{1/2,p}}  \|v\|_{M^{0,p'}}.
 \end{split}
 \end{equation*}
 
 \noindent
{\bf Sub-Case IIb iii) $\boxed{0<s<1/2}$:} Again using \eqref{bt5}, H\"older's inequality, the estimate \eqref{bt4.2} yields
\begin{equation*}
 \begin{split}
 \mathcal{Y}
 \lesssim &\| v\|_{M^{0,p'}} \, \sum_{N_1 \geq 1} N_1^{-s}\|u_{N_1}\|_{H^s}\sum_{ N_2\lesssim N_1}  N_2^{-s}\|u_{N_2}\|_{H^s} \\
 \lesssim & \| v\|_{M^{0,p'}} \, \sum_{N_1 \geq 1} N_1^{-s}\|u_{N_1}\|_{H^s}\Big(\sum_{ N_2\lesssim N_1}  \dfrac1{N_2^{sp'}}\Big)^{1/p'}\|u_{2}\|_{M^{s,p}} \\
 \lesssim & \| v\|_{M^{0,p'}} \|u_{2}\|_{M^{s,p}} \|u_{1}\|_{M^{s,p}}\, \left(\sum_{N_1 \geq 1} \dfrac1{N_1^{sp'}}\right)^{1/p'}\\
 \lesssim & \| v\|_{M^{0,p'}} \|u_{2}\|_{M^{s,p}} \|u_{1}\|_{M^{s,p}}.
 \end{split}
 \end{equation*}
 
 Now, we move on to prove \eqref{bt} considering the case $p=1$.
 This case is easier or with the same difficulty as the previous case. In fact, considering the first inequality in \eqref{bt4.1} and applying H\"older's inequality with $p=1$ and $p'=\infty$, we have
 \begin{equation*}
 \begin{split}
 \mathcal{X} & \lesssim \sum_{N \gg 1} \dfrac{\|v_{N}\|_{L^2}}{N}\sum_{ N_1\lesssim N} \|u_{N_1}\|_{H^s} \left(\| u_{2}\|_{M^{s,1}} N^{1/2-s}\right)\\
& \lesssim  \| u_{2}\|_{M^{s,1}} \sum_{N \gg 1} \dfrac{\|v_{N}\|_{L^2}}{N^{1/2+s}}\sum_{ N_1\lesssim N} \|u_{N_1}\|_{H^s} \\
& \lesssim  \| u_{2}\|_{M^{s,1}} \| u_{1}\|_{M^{s,1}} \sum_{N \gg 1} \dfrac{\|v_{N}\|_{L^2}}{N^{1/2+s}} \\
& \lesssim  \| u_{2}\|_{M^{s,1}} \| u_{1}\|_{M^{s,1}} \|v\|_{M^{s,\infty}} \sum_{N \gg 1} \dfrac{1}{N^{1/2+s}}\\
 & \lesssim  \| u_{2}\|_{M^{s,1}} \| u_{1}\|_{M^{s,1}} \|v\|_{M^{s,\infty}}.
 \end{split}
 \end{equation*}
 
 Now, considering first inequality in \eqref{bt4.2} in the worst case when $s>\frac12$, applying H\"older's inequality with $p=1$ and $p'=\infty$ and Lemma \ref{lem32}, we obtain
 \begin{equation*}
 \begin{split}
 \mathcal{Y}
 & \lesssim \sum_{N_1 \gg 1} N_1^{-s}\|u_{N_1}\|_{H^s}\sum_{ N_2\lesssim N_1} N_2^{-s}\|u_{N_2}\|_{H^s} \Big(\| v\|_{M^{s,\infty}} \sum_{ N\lesssim N_1} \dfrac{1}{N^{1/2-s}}\Big)
 \\
 &\lesssim \| v\|_{M^{s,\infty}} \sum_{N_1 \gg 1} N_1^{-s}\|u_{N_1}\|_{H^s}\sum_{ N_2\lesssim N_1} N_2^{-s}\|u_{N_2}\|_{H^s} N_1^{s-1/2}\\
& \lesssim \| v\|_{M^{s,\infty}} \sum_{N_1 \gg 1} N_1^{-1/2}\|u_{N_1}\|_{H^s}\sum_{ N_2\lesssim N_1} N_2^{-s}\|u_{N_2}\|_{H^s}\\
& \lesssim \| v\|_{M^{s,\infty}} \| u_{2}\|_{M^{s,1}}\sum_{N_1 \gg 1} N_1^{-1/2}\|u_{N_1}\|_{H^s}\\
& \lesssim \| v\|_{M^{s,\infty}} \| u_{2}\|_{M^{s,1}}\| u_{1}\|_{M^{s,1}}.
 \end{split}
 \end{equation*}

This completes the proof  of Lemma.
\end{proof}
 
 \begin{remark}
  Observe that the end point case $s=0$ holds in all cases expect in  the {\bf Sub-Case IIb iii)}.
  \end{remark}

\begin{lemma}\label{Lema1}
For any $s > 0$ and $p\geq 1$, there is a constant $C = C_s$ such that the inequality
\begin{equation}\label{bilin-1}
\|\tau(\partial_x) \eta^2\|_{M^{s,p}} \le C \| \eta\|_{M^{s,p}}^2
\end{equation}
holds, where the operator $\tau(\partial_x)$  is  defined in (\ref{phi-D}).
\end{lemma}

\begin{proof}
Since  $\delta_1>0$, one can easily verify that  $|\tau(\xi)| \lesssim \frac{|\xi|}{1+\xi^2}=\omega(\xi)$.   Using this fact and the definition of the $M^{ s,p}$-norm along with Lemma \ref{BT1}, one can obtain
\begin{equation*}
\begin{split}
\|\tau(\partial_x) \eta^2\|_{M^{s,p}}  \lesssim \|\omega(\partial_x) \eta^2\|_{M^{s,p}}  \lesssim\| \eta\|_{M^{s,p}}^2,
\end{split}
\end{equation*}
as required.
\end{proof}


\begin{lemma}\label{P1}
For $s \ge 1$ and $p\geq 1$ there is a constant $C = C_s$ such that 
\begin{equation}\label{trilin-1}
\|\psi(\partial_x) \eta^3\|_{M^{s,p}} \le C \| \eta\|_{M^{s,p}} ^3.
\end{equation}
\end{lemma}
\begin{proof}
As in the proof of Lemma \ref{BT1}, the estimate \eqref{trilin-1} is a consequence of the following inequality
\begin{equation}\label{ab1}
\int_{\xi_1+\xi_2+\xi_3+\xi=0}\langle \xi \rangle^{s-3}\widehat{v}(\xi)\widehat{u_1}(\xi_1) \widehat{u_2}(\xi_2) \widehat{u_3}(\xi_3) 
\lesssim \|u_1\|_{M^{s,p}}\|u_2\|_{M^{s,p}}\|u_3\|_{M^{s,p}}\|v\|_{M^{0,p'}}.
\end{equation}

Let $N_1$, $N_2$, $N_3$ and $N$ be dyadic numbers such that
$$
|\xi_j|\sim N_j, j=1,2,3\quad |\xi|\sim N.
$$
Let 
$$
|\xi_{\textrm max}|= \max \{|\xi_1|, |\xi_2|, |\xi_3|\}, \quad |\xi_{\textrm min}|= \min \{|\xi_1|, |\xi_2|, |\xi_3|\},
$$
and $|\xi_{\textrm med}|$ be such that $|\xi_{\textrm max}|\geq |\xi_{\textrm med}| \geq |\xi_{\textrm min}|$. We denote by $N_{\max}\sim |\xi_{\textrm max}|$, $N_{\textrm{med}}\sim |\xi_{\textrm med}|$ and $N_{\min}\sim |\xi_{\textrm min}|$.  
 Since  $\xi=-\xi_1-\xi_2-\xi_3$ one has $|\xi|\leq 3|\xi_{\max}|$. 

Without loss of generality we can suppose that $N_{\max}=N_1\geq N_{\textrm{med}}=N_2\geq N_{\min}=N_3$. Thus, we have $N\lesssim N_1$.

In sequel we prove \eqref{ab1} considering four different cases.\\

\noindent
{\bf Case I  (Trivial Case). $\boxed{N_{\max} \leq 1}$:}
In this case,  using $u_N=P_Nu$ and Plancherel's inequality, to obtain \eqref{ab1} we need to estimate the following expression
 \begin{equation}\label{X1}
 \mathcal{X}_1:=\sum_{N, N_1 \lesssim 1} \int N^{s-3} u_{N_1}u_{N_2} u_{N_3}v_{N}.
 \end{equation}
 
 Using H\"older's inequality and Bernstein's inequality \eqref{bern-1}, we obtain from \eqref{X1} that
 \begin{equation}\label{ab2}
 \begin{split}
 \mathcal{X}_1
 \leq &\sum_{N, N_{\max} \lesssim 1}N^{s-3} \|u_{N_1}\|_{L^2}N_2^{1/2}\| u_{N_2}\|_{L^2} N_3^{1/2}\| u_{N_3}\|_{L^2}\|v_{N}\|_{L^2}\\
 \lesssim &\sum_{N, N_{\max} \lesssim 1}N^{s-3}N_1 N_1^{s}\|u_{N_1}\|_{L^2}N_2^{s}\| u_{N_2}\|_{L^2} N_3^{s}\| u_{N_3}\|_{L^2}\|v_{N}\|_{L^2}N_1^{-s}N_2^{-s}N_3^{-s}\\
 \lesssim &\sum_{N, N_{\max} \lesssim 1}\|u_{N_1}\|_{H^s}\| u_{N_2}\|_{H^s} \| u_{N_3}\|_{H^s} \|v_{N}\|_{L^2}.
 \end{split}
 \end{equation}
 
 Now, opening the sum in \eqref{ab2}, using H\"older's inequality and Lemma \ref{lem32}, we have
 \begin{equation}\label{1xbt4}
 \begin{split}
 \mathcal{X}_1
 \lesssim &\sum_{N \lesssim 1} \|v_{N}\|_{L^2}\sum_{N_1 \lesssim 1}\|u_{N_1}\|_{H^s}\sum_{ N_2\lesssim 1}  \| u_{N_2}\|_{H^s}\sum_{ N_3\leq 1} \| u_{N_3}\|_{H^s}
  \\
   \lesssim &\| u_3\|_{M^{s,p}}\| u_2\|_{M^{s,p}} \|u_1\|_{M^{s,p}}  \|v\|_{M^{0,p'}},
 \end{split}
 \end{equation}
since the sums in \eqref{1xbt4} are finite. 

From now on, we will consider that $N_{\max}\gg1$ to deal with the following cases.\\

\noindent
{\bf Case II (Non-resonant case). $\boxed{N_{\max} \gg N_{\mathrm{med}}}$:} In this case we have $N \sim N_{\max}$. So, to obtain \eqref{ab1} we need to estimate the expression
 \begin{equation}\label{X2}
 \mathcal{X}_2:=\sum_{N\sim N_1 \gg N_2\geq N_3} \int N^{s-3} u_{N_1}u_{N_2} u_{N_3}v_{N}.
 \end{equation}
 
 Using H\"older's inequality and Bernstein's inequality \eqref{bern-1}, we obtain from \eqref{X2} that
 \begin{equation}\label{ab3}
 \begin{split}
 \mathcal{X}_2
 \leq &\sum_{N\sim N_1 \gg N_2\geq N_3}N^{s-3} \|u_{N_1}\|_{L^2}N_2^{1/2}\| u_{N_2}\|_{L^2} N_3^{1/2}\| u_{N_3}\|_{L^2}\|v_{N}\|_{L^2}\\
 \lesssim &\sum_{N\sim N_1 \gg N_2\geq N_3}N^{s-3}N N_1^{s}\|u_{N_1}\|_{L^2}N_2^{s}\| u_{N_2}\|_{L^2} N_3^{s}\| u_{N_3}\|_{L^2}\|v_{N}\|_{L^2}N_1^{-s}N_2^{-s}N_3^{-s}\\
 \lesssim &\sum_{N\sim N_1 \gg N_2\geq N_3}N^{-2}N_2^{-s}N_3^{-s}\|u_{N_1}\|_{H^s}\| u_{N_2}\|_{H^s} \| u_{N_3}\|_{H^s} \|v_{N}\|_{L^2}.
 \end{split}
 \end{equation}
 
 Now, opening the sum in \eqref{ab3}, using H\"older's inequality and Lemma \ref{lem32}, we have
 \begin{equation*}
 \begin{split}
 \mathcal{X}_2
 \lesssim &\sum_{N } N^{-2}\|v_{N}\|_{L^2}\sum_{N_1 \lesssim N}\|u_{N_1}\|_{H^s}\sum_{ N_2\lesssim N}  N_2^{-s}\| u_{N_2}\|_{H^s}\sum_{ N_3\leq N}N_3^{-s} \| u_{N_3}\|_{H^s}\\
   \lesssim &\| u_3\|_{M^{s,p}}\| u_2\|_{M^{s,p}} \|u_1\|_{M^{s,p}} \sum_{N } N^{-2}(\ln N)^{1/p'}\|v_{N}\|_{L^2}\\
   \lesssim &\| u_3\|_{M^{s,p}}\| u_2\|_{M^{s,p}} \|u_1\|_{M^{s,p}} \|v\|_{M^{0,p'}}\left(\sum_{N } N^{-2p}(\ln N)^{p/p'}\right)^{1/p}\\
    \lesssim &\| u_3\|_{M^{s,p}}\| u_2\|_{M^{s,p}} \|u_1\|_{M^{s,p}} \|v\|_{M^{0,p'}}.
 \end{split}
 \end{equation*}
 
 \noindent
{\bf Case III (Semi-resonant case). $\boxed{N_{\max} \sim N_{\textrm{med}}\gg N_{\min}}$:} We analyse this case further dividing in two sub-cases.\\

\noindent
{\bf Sub-Case IIIa. $\boxed{N \ll N_{\max}}:$} In this case we need to estimate
 \begin{equation}\label{X3}
 \mathcal{X}_3:=\sum_{N_{\max} \sim N_{\textrm{med}}\gg N_{\min}} \int N^{s-3} u_{N_1}u_{N_2} u_{N_3}v_{N}.
 \end{equation}
 
 Using H\"older's inequality and Bernstein's inequality \eqref{bern-1} in \eqref{X3}, we obtain
 \begin{equation}\label{ab4}
 \begin{split}
 \mathcal{X}_3
 \leq &\sum_{N_1\sim N_2 \gg N_3,N}N^{s-3} \|u_{N_1}\|_{L^2}N_2^{1/2}\| u_{N_2}\|_{L^2} N_3^{1/2}\| u_{N_3}\|_{L^2}\|v_{N}\|_{L^2}\\
 \lesssim &\sum_{N_1\sim N_2 \gg N_3,N}N^{s-3}N_1^{-s}N_2^{1/2-s}N_3^{1/2-s}\|u_{N_1}\|_{H^s}\| u_{N_2}\|_{H^s} \| u_{N_3}\|_{H^s} \|v_{N}\|_{L^2}.
 \end{split}
 \end{equation}
 
If $s>3$, opening the sum in \eqref{ab4}, using H\"older's inequality and Lemma \ref{lem32}, we have
 \begin{equation*}
 \begin{split}
 \mathcal{X}_3
 \lesssim &\sum_{N_1 } N_1^{-s}\|u_1\|_{M^{s,p}}  \sum_{N_2 \lesssim N_1}N_2^{1/2-s}\|u_{N_2}\|_{H^s}\sum_{ N\lesssim N_1}  N^{s-3}\|v_{N}\|_{L^2}\sum_{ N_3\leq N_1}N_3^{1/2-s} \| u_{N_3}\|_{H^s}\\
   \lesssim &\| u_3\|_{M^{s,p}}\| v\|_{M^{0,p'}} \sum_{N_1 } N_1^{-s}N_1^{s-3}\|u_1\|_{M^{s,p}}  \sum_{N_2 \lesssim N_1}N_2^{1/2-s}\|u_{N_2}\|_{H^s}\\
   \lesssim &\| u_3\|_{M^{s,p}}\|v\|_{M^{0,p'}}\| u_2\|_{M^{s,p}} \|u_1\|_{M^{s,p}} \Big(\sum_{N_1 } N_1^{-3p'}\Big)^{1/p'}\\
    \lesssim &\| u_3\|_{M^{s,p}}\|v\|_{M^{0,p'}}\| u_2\|_{M^{s,p}} \|u_1\|_{M^{s,p}}.
 \end{split}
 \end{equation*}
 
If $1<s\leq 3$, we can use  Lemma \ref{lem32} in $\sum_{ N\lesssim N_1}  N^{s-3}\|v_{N}\|_{L^2}$ and the required result follows.\\

\noindent
{\bf Sub-Case IIIb.  $\boxed{N \sim N_{\max}}$:}
In this case we have $N_1\sim N_2\sim N\gg N_3$. This case is similar to the the Case II so we omit the details.\\

\noindent
{\bf Case IV  (Resonant case). $\boxed{N_{\max} \sim  N_{\min}}$:}
In this case we have
$$
N\lesssim N_1\sim N_2\sim N_3, 
$$
 and we  proceed as in {\bf Sub-Case IIIa} and this completes the proof in all possible cases.
\end{proof}


\begin{lemma}
For $s > 1$, the following estimate
\begin{equation}\label{Sharp1}
\|\psi(\partial_x) \eta_x^2\|_{M^{s,p}} \le C \| \eta\|_{M^{s,p}}^2
\end{equation}
holds.
\end{lemma}
\begin{proof}
Observe that
$$
\psi(\xi) \lesssim \frac{\omega(\xi) }{\langle \xi \rangle}.
$$
The inequality (\ref{bt})  then allows the conclusion 
\begin{equation*}
\|\psi(\partial_x) \eta_x^2\|_{M^{s,p}}  \lesssim \| \omega(\partial_x) \eta_x^2\|_{M^{s-1,p}} \lesssim \|  \eta_x\|_{M^{s-1, p}}\|  \eta_x\|_{M^{s-1, p}}
\lesssim  \|  \eta\|_{M^{s,p}}^2,
\end{equation*}
since  $s-1 > 0$.
\end{proof}

The preceding ingredients are now assembled to provide a proof of the local well-posedness result stated in Theorem \ref{mainTh1}.

\begin{proof}[Proof of Theorem \ref{mainTh1}] Let $s>1$, $p\geq 1$ and $\eta_0\in M^{s,p}$ be given.
 Define a mapping
\begin{equation}\label{eq3.42}
\Psi\eta(x,t) = S(t)\eta_0 -i\int_0^tS(t-t')\Big(\tau(D_x)\eta^2 - \frac14 \psi(\partial_x)\eta^3
 -\frac7{48} \psi(\partial_x)\eta_x^2\Big)(x, t') dt'.
\end{equation}
The immediate goal is to show that this mapping is a contraction on a closed ball $\mathcal{B}_r$ with radius $r > 0$ and center at the origin in $C([0,T];M^{s,p})$.

As remarked earlier,  $S(t)$ is a unitary group in $M^{s,p}(\R)$  (see (\ref{eq1.11})), and therefore
\begin{equation*}
\|\Psi\eta\|_{M^{s,p}} \leq \|\eta_0\|_{M^{s,p}} +CT\Big[\big{\|}\tau(\partial_x)\eta^2 - \frac18 \psi(\partial_x)\eta^3
 -\frac7{48}\psi(\partial_x)\eta_x^2\big{\|}_{C([0,T];M^{s,p})}\Big].
\end{equation*}
The inequalities (\ref{bilin-1}), (\ref{trilin-1}) and (\ref{Sharp1}) lead immediately to 
\begin{equation}\label{eq3.44}
\|\Psi\eta\|_{M^{s,p}} \leq \|\eta_0\|_{M^{s,p}} +CT\Big[\big{\|}\eta\big{\|}_{C([0,T];M^{s,p})}^2 + \big{\|}\eta\big{\|}_{C([0,T];M^{s,p})}^3 +\big{\|}\eta\big{\|}_{C([0,T];M^{s,p})}^2\Big].
\end{equation}
If, in fact,  $\eta\in \mathcal{B}_r$, then (\ref{eq3.44}) yields
\begin{equation*}
\|\Psi\eta\|_{M^{s,p}} \leq \|\eta_0\|_{M^{s,p}} +CT\big[2r +r^2 \big]r.
\end{equation*}

If we choose $r= 2\|\eta_0\|_{H^s}$ and $T= \frac1{2Cr(2 + r) }$, then  $\|\Psi\eta\|_{M^{s,p}} \leq r$, showing that $\Psi$ maps
 the closed ball $\mathcal{B}_r$ in $C([0,T];M^{s,p})$ onto itself.   With the same choice of $r$ and $T$ and the same sort of 
estimates, one discovers that $\Psi$ is a contraction on $\mathcal{B}_r$ with contraction constant equal to $\frac12$ as it happens. The rest of
the proof is standard.
\end{proof}

\begin{remark}\label{rm2.1}
The following points follow immediately from the proof of the Theorem \ref{mainTh1}:
\begin{enumerate}
\item The maximal existence time $T_s$ of the solution satisfies
\begin{equation}\label{r2.45}
T_s\geq \bar{T} = \frac1{8C_s\|\eta_0\|_{M^{s,p}}(1+\|\eta_0\|_{M^{s,p}})},
\end{equation}
where the constant $C_s$ depends only on $s$.
\item The solution cannot grow too much on the interval $[0,\bar T]$ since
\begin{equation}\label{r2.46}
\|\eta(\cdot,t)\|_{M^{s,p}} \leq r =  2\|\eta_0\|_{M^{s,p}}
\end{equation}
for  $t $ in this interval,  where $\bar{T}$ is as above in (\ref{r2.45}).
\end{enumerate}
\end{remark}

\section{Global well-posedness in $M^{s,p}$ spaces}\label{sec-4}

In this section, we will use the splitting argument introduced in \cite{B1, B} and way earlier in \cite{BS} to get the global solution of the IVP \eqref{5kdvbbm} for given data in the modulation space stated in Theorem \ref{Th-global}. For this, we proceed as follows.

Let $\eta_0 \in M^{s,p}$, $s\geq 1$, $p\geq 1$ and  $N$ be a large number to be chosen later.
We split the initial data $\eta_0=u_0+v_0$ where  $\widehat{u_0}=\widehat{\eta_0} \chi_{\{|\xi| \leq N\}}$ and  $\widehat{v_0}=\widehat{\eta_0} \chi_{\{|\xi| > N\}}$. 
Using H\"older's inequality and \eqref{immersion1}, we have, for any $\kappa \in \R$ and $p\geq 2$
\begin{equation}\label{est-4.1}
\|u_0\|_{{M}^{\kappa, p}} \lesssim \|u_0\|_{H^{\kappa}} \lesssim \|u_0\|_{{M}^{\kappa, p}} N^{\frac12-\frac1p}.
\end{equation}

It  is easy to verify that $ u_0 \in M^{\delta,p}$ for any  $\delta \geq s$ and $v_0 \in M^{s,p}$. 
In fact, we have
\begin{equation}\label{31}
\begin{cases}
\|u_0\|_{L^2}\leq \|\eta_0\|_{L^2},\\
\|u_0\|_{{M}^{\delta, p}}\lesssim \|\eta_0\|_{{M}^{s,p}}\,N^{\delta-s},\qquad \delta\geq s, \;\; p\geq 1,\\
\|u_0\|_{H^{\delta}}\lesssim \|\eta_0\|_{{M}^{s,p}}\, N^{\frac12-\frac1p}N^{\delta-s},\quad \delta\geq s,\quad p\geq 2.
\end{cases}
\end{equation}
and
\begin{equation}\label{32}
\begin{cases}
\|v_0\|_{M^{\rho,p}}\lesssim\|\eta_0\|_{M^{s,p}}\,N^{(\rho-s)}, \qquad 0\leq\rho\leq s,\\
\|v_0\|_{H^{\rho}} \lesssim\|\eta_0\|_{M^{s,p}}\,N^{(\rho-s)}, \qquad s-\rho>\frac12-\frac1p,
\end{cases}
\end{equation}
where in the last estimate H\"older's inequality was used.

For the low frequency part  $u_0$  of $\eta_0$ we associate the IVP
\begin{equation}\label{xeq1}
\begin{cases}
iu_t = \phi(\partial_x)u +F(u),\\
 u(x,0) = u_0(x),
 \end{cases}
\end{equation}
where $F(u)= \tau (\partial_x)u^2 - \frac18\psi(\partial_x)u^3  -\frac7{48}\psi(\partial_x)u_x^2\,$ and for the high frequency part  $
v_0$  of $\eta_0$ we associate the IVP
\begin{equation}\label{xeq2}
\begin{cases}
iv_t = \phi(\partial_x)v + F(u+v)-F(u).\\
 v(x,0) = v_0(x),
 \end{cases}
\end{equation}

Observe that, if $u$ and $v$ are   solutions to the IVPs \eqref{xeq1} and \eqref{xeq2} respectively, then $\eta(x,t)=u(x,t)+v(x,t)$ solves the original IVP \eqref{eq1.9} in the common time interval of existence of $u$ and $v$. Taking in consideration this observation, first we will prove that there exists a time $T_u$ such that the IVP \eqref{xeq1} is locally well-posed in $[0, T_u]$. Now, fixing the solution $u$ of the IVP \eqref{xeq1}, we prove that there exists a time $T_v$ such that the IVP \eqref{xeq2} is locally well-posed in $[0, T_v]$. Hence, if we consider $t_0\leq\min\{T_u, T_v\}$, then  $\eta=u+v$ solves the IVP \eqref{eq1.9} in the time interval $[0, t_0]$ for given data in $M^{s, p}$, $s> 1$, $p\geq 2$.  This restriction $p\geq 2$ is necessary to control the existence time $t_0$ using estimate involving energy, see \eqref{xeq5} below. For $1<s<2$, our objective is to iterate this process  maintaining the common local existence time $t_0$  in each iteration  to cover any given time interval $[0, T]$ there by getting global well-posedness result (see {\bf Case 1} in the proof of Theorem \ref{Th-global} below). However, for $s\geq 2$, we derive a uniform {\em a priori} estimate for $\|\eta\|_{M^{s,p}}$-norm  in terms of $\|M^{s_0, p}$-norm for some $1\leq\frac32-\frac1p\leq s_0 <2$ fixed, and use it to get the required global  result (see {\bf Case 2} in the proof of Theorem \ref{Th-global} below)..

Notice first that, from Theorem \ref{mainTh1} the IVP  \eqref{xeq1} is locally well-posed in $M^{s,p}$, $s > 1$ with existence time given by $T_{u} =\dfrac{c_s}{\|u_0\|_{M^{s,p}}(1+\|u_0\|_{M^{s,p}})}$ (see \eqref{r2.45} in Remark \ref{rm2.1}).  Now, fixing the local solution of the IVP \eqref{xeq1}, we consider the IVP \eqref{xeq2} with variable coefficients that depend on $u$, and prove  the following local well-posedness result.

\begin{theorem}\label{mainTh4}  Let $\gamma_1, \delta_1 >0$ and $u$ be the solution to the IVP \eqref{xeq1}.  
For any $s> 1$ and   $v_0\in M^{s,p}(\R)$, there exist a time $$T_v =\dfrac{c_s}{(\|v_0\|_{M^{s,p}}+\|u_0\|_{M^{s,p}})(1+\|v_0\|_{M^{s,p}}+\|u_0\|_{M^{s,p}})}$$
 and a unique function  $v \in C([0,T_v];M^{s,p})$ which 
 is a solution of the IVP \eqref{xeq2}.  The solution $v$ 
varies continuously in $C([0,T_v];M^{s,p})$ as $v_0$ varies in $M^{s,p}$.
\end{theorem} 
\begin{proof}
Using  Duhamel's formula, the equivalent integral equation to \eqref{xeq2} is
\begin{equation}\label{xeq6}
\begin{split}
v(x,t) &= S(t)v_0 -i\int_0^tS(t-t')\Big( F(u+v)-F(u)\Big)(x, t') dt'\\
&=: S(t)v_0 + h(x,t),
 \end{split}
\end{equation}
where 
\begin{equation}\label{xeq7}
 F(u+v)-F(u)= \tau (\partial_x)(v^2+2vu) - \frac18\psi(\partial_x)(3u^2v+3uv^2+v^3) -\frac7{48}\psi(\partial_x)(2u_xv_x+v_x^2).
\end{equation}

Let $u\in C([0,T_u];M^{s,p})$ be the solution of IVP \eqref{xeq1}, given by Theorem \ref{mainTh1}. From \eqref{r2.46}, we have that $u$ satisfies
\begin{equation}\label{eqofu}
\sup_{t \in [0, T_u]}\|u(t)\|_{M^{s,p}} \lesssim \|u_0\|_{M^{s,p}}.
\end{equation} 

Let  $a:=2\|v_0\|_{M^{s,p}}$ and define a ball in $ C([0,T];M^{s,p}) $
$$
X_T^a:= \{ v \in C([0,T];M^{s,p}) :   \quad |||v|||  \leq a \}
$$
where, $ |||v|||:=\sup_{t \in [0, T]}\|v(t)\|_{M^{s,p}}$. Now, define an application
$$
\Phi_u(v)(x,t) = S(t)v_0 -i\int_0^tS(t-t')\Big( F(u+v)-F(u)\Big)(x, t') dt'.
$$
We will show that, there exists a time $T>0$ such that the  application $\Phi_u(v)$ is a contraction on the on the ball $X_T^a$.

Let $T\leq T_u$, then using the estimate for the unitary group  $S(t)$ in $M^{s,p}(\R)$, in the light of \eqref{xeq7}, we obtain
\begin{equation}\label{eq-v22}
\begin{split}
\|\Phi_u(v)\|_{M^{s,p}}& \leq \|v_0\|_{M^{s,p}}\\ 
&\quad +T |||\tau (\partial_x)(v^2+2vu) - \frac18\psi(\partial_x)(3u^2v+3uv^2+v^3) 
-\frac7{48}\psi(\partial_x)(2u_xv_x+v_x^2)|||.
\end{split}
\end{equation}

Using inequalities  \eqref{bilin-1}, \eqref{trilin-1}, \eqref{Sharp1} and \eqref{eqofu}, the estimate \eqref{eq-v22} yields
\begin{equation}\label{xeq55}
\begin{split}
\|\Phi_u(v)\|_{M^{s,p}} &\leq \|v_0\|_{M^{s,p}} \\
&\quad +T |||\tau (\partial_x)(v^2+2vu) - \frac18\psi(\partial_x)(3u^2v+3uv^2+v^3) -\frac7{48}\psi(\partial_x)(2u_xv_x+v_x^2)||| \\
&\leq   \frac{a}2+cT|||v||| (|||v|||+ \|u_0\|_{M^{s,p}}) +cT|||v||| (\|u_0\|_{M^{s,p}}^2+\|u_0\|_{M^{s,p}} |||v|||+ |||v|||^2)
\\
&\leq   \frac{a}2+cT[a (a+ \|u_0\|_{M^{s,p}}) (1+a+ \|u_0\|_{M^{s,p}})].
\end{split}
\end{equation}

Now, choosing 
\begin{equation*}
cT[(a+ \|u_0\|_{M^{s,p}}) (1+a+ \|u_0\|_{M^{s,p}})] = \frac12
\end{equation*}
the estimate \eqref{xeq55}  readily yields  $\|\Phi_u(v)\|_{M^{s,p}} \leq a$ thereby showing that $\Phi_u(v)$ is a self map on
 the closed ball $X_T^a$ in $C([0,T];M^{s,p})$.   With a similar argument  one can show that the application  $\Phi_u(v)$ is also a contraction on $X_T^a$. The rest of the proof is standard, so we omit the details.
\end{proof}

The following result will be useful in our argument in proving global well-posedness result when $1< s<2$.

\begin{lemma}\label{lemah1}
Let $1< s<2$,  $u$ be the solution of the IVP \eqref{xeq1} and $v$ be the solution of the IVP \eqref{xeq2}, both with initial data in $M^{s,p}(\R)$. Then, for  $s>\frac32-\frac1p$, the function $h=h(u,v)$  defined in \eqref{xeq6} is in $C([0,t_0], H^2(\R))$ and satisfies
\begin{equation}\label{estim1}
\|u(t_0)\|_{H^2} \lesssim N^{\vartheta} \quad \textrm{and}\quad \|h(t_0)\|_{H^2} \lesssim N^{s-3},
\end{equation} 
where $t_0 \sim N^{-2\vartheta}$, $\vartheta=(2-s)+\frac12-\frac1p$.
\end{lemma}

\begin{proof}
The proof of this result follows from the proof of Lemma 2.6 in \cite{CP} by noticing from \eqref{est-4.1} and \eqref{32} that, in this case, one has
$$\|u\|_{H^1}\lesssim N^{\frac12-\frac1p}, \qquad {\mathrm{and}}\qquad \|v\|_{H^1}\lesssim N^{1-s},$$
since $\|u_0\|_{H^\theta}\lesssim \|u_0\|_{M^{\theta, p}}N^{\frac12-\frac1p} \leq \|\eta_0\|_{M^{s, p}}N^{\frac12-\frac1p}$, for any $\theta\leq s$ and $p\geq 2$. So, we omit the details.
\end{proof}

The following result will play a crucial role while dealing with higher Sobolev regularity data.
 \begin{lemma}\label{estimnormgwp}
  Let $p\geq 1$, then there is a constant $C = C_s$ such that the following estimates hold
\begin{equation}\label{xavbt1}
\|\tau(\partial_x) (\eta^2)\|_{M^{s_2, p}} \le C\|\eta\|_{M^{s_1, p}}\|\eta\|_{M^{s_2, p}}, \quad {\mathrm {for}}\;\; s_2 \geq s_1>0,
\end{equation}

\begin{equation}\label{xavbt2}
\|\psi(\partial_x) (\eta^3)\|_{M^{s_2, p}} \le C\|\eta\|_{M^{s_1, p}}^2 \|\eta\|_{M^{s_2, p}}, \quad {\mathrm {for}}\;\; s_2 \geq s_1\geq 1
\end{equation}
and 
\begin{equation}\label{xavbt3}
\|\psi(\partial_x) (\eta_x)^2 \|_{M^{s_2, p}} \le C\|\eta\|_{M^{s_1, p}}\|\eta\|_{M^{s_2, p}}, \quad {\mathrm {for}}\;\; s_2 \geq s_1>1.
\end{equation}
 \end{lemma}
 \begin{proof}
 Let $s:=s_2-s_1$. First, notice that is suffices prove \eqref{xavbt1}  considering $\omega(\partial_x)$ instead of $\tau(\partial_x)$.
By \eqref{bt3}, we need to prove
 \begin{equation}\label{xavbt4}
\int_{\xi_1+\xi_2+\xi=0}\dfrac{\langle \xi \rangle^{s_2}}{\langle \xi \rangle}\widehat{v}(\xi)\widehat{\eta}(\xi_1) \widehat{\eta}(\xi_2)  
\lesssim \|\eta\|_{M^{s_1,p}}\|\eta\|_{M^{s_2, p}}\|v\|_{M^{0,p'}}.
\end{equation}
Without loss of generality we can suppose that $|\xi_2|\leq |\xi_1|$, so that $|\xi|\leq 2 |\xi_1|$. Using \eqref{bt3} with $\widehat{u_1}(\xi_1)=\langle \xi_1 \rangle^{s}\widehat{\eta}(\xi_1) $ and $\widehat{u_2}(\xi_2)=\widehat{\eta}(\xi_2)$, we have
 \begin{equation*}
\begin{split}
\int_{\xi_1+\xi_2+\xi=0}\dfrac{\langle \xi \rangle^{s+s_1}}{\langle \xi \rangle}\widehat{v}(\xi)\widehat{\eta}(\xi_1) \widehat{\eta}(\xi_2) 
&\lesssim \int_{\xi_1+\xi_2+\xi=0}\dfrac{\langle \xi \rangle^{s_1}}{\langle \xi \rangle}\widehat{v}(\xi) \left[\langle \xi_1 \rangle^{s}\widehat{\eta}(\xi_1) \right]\widehat{\eta}(\xi_2)\\
&\lesssim \|\eta\|_{M^{s_1,p}}\|\eta\|_{M^{s_2, p}}\|v\|_{M^{0,p'}},
\end{split}
\end{equation*}
as required in \eqref{xavbt4}.

In order to prove \eqref{xavbt3}, we use that $\psi(\xi)\lesssim \frac{\omega(\xi)}{\langle \xi \rangle}$ and \eqref{xavbt1}, where $s_2-1 \geq s_1-1>0$. With these considerations
\begin{equation*}
\|\psi(\partial_x) (\eta_x)^2 \|_{M^{s_2, p}} \lesssim \|\omega(\partial_x) (\eta_x)^2 \|_{M^{{s_2-1}, p}}  \lesssim \|\eta_x\|_{M^{{s_1-1}, p}}\|\eta_x\|_{M^{{s_2-1}, p}} \lesssim \|\eta\|_{M^{{s_1}, p}}\|\eta\|_{M^{{s_2}, p}},
\end{equation*}
as announced in \eqref{xavbt3}.

Now, we move to prove \eqref{xavbt2}. Without loss of generality, one can suppose that 
$$|\xi_1|=\max\{|\xi_1|,|\xi_2|, |\xi_3|\},$$
so that $|\xi|\leq 3 |\xi_1|$. In this way
\begin{equation}\label{xavbt7}
\begin{split} 
\int_{\xi_1+\xi_2+\xi_3+\xi=0}\dfrac{\langle \xi \rangle^{s+s_1}}{\langle \xi \rangle^3}\widehat{v}(\xi)\widehat{\eta}(\xi_1) \widehat{\eta}(\xi_2) \widehat{\eta}(\xi_3) 
&\lesssim \int_{\xi_1+\xi_2+\xi_3+\xi=0}\dfrac{\langle \xi \rangle^{s_1}}{\langle \xi \rangle^3}\widehat{v}(\xi) \left[ \langle \xi_1 \rangle^{s}\widehat{\eta}(\xi_1) \right] \widehat{\eta}(\xi_2) \widehat{\eta}(\xi_3)\\
&\lesssim \|\eta\|_{M^{s_1,p}}^2\|\eta\|_{M^{s_2, p}}\|v\|_{M^{0,p'}}.
\end{split}
\end{equation}
where in the last inequality we  used \eqref{ab1}. From \eqref{xavbt7} one easily obtains the required estimate \eqref{xavbt2}.
\end{proof}

Now we are in position to supply the proof of the global well-posedness result.

\begin{proof}[Proof of Theorem \ref{Th-global}]  Let $\eta_0\in M^{s,p}$, with $p\geq 2$ and $s>1$. As discussed earlier, our idea is to extend the local solution given by Theorem \ref{mainTh1} to the time interval $[0, T]$ for any $T\gg 1$.  As discussed above, we split the initial data $\eta_0=u_0+v_0$ so that $u_0$ and $v_0$ satisfy the growth conditions \eqref{31} and \eqref{32} respectively.

Let $u$ be the solution  to the IVP \eqref{xeq1} given by Theorem \ref{mainTh1} with data  $u_0$ and $v$ be the solution to the  IVP \eqref{xeq2} given by Theorem \ref{mainTh4} with data $v_0$. Note  that, the sum $\eta=u+v$ solves the original IVP \eqref{eq1.9} in the common time interval of existence of $u$ and $v$.

 We divide the proof in two different cases, viz.,  $1< s< 2$ and $s\geq 2$.\\

\noindent
{\bf Case 1. $\boxed{1<s<2}$:} Observe that from \eqref{energy} and  \eqref{31},  we have
\begin{equation}\label{xeq4}
E(u(t))=E(u_0) \sim\|u_0\|_{H^2}^2 \lesssim N^{2\vartheta},
\end{equation}
where $\vartheta=(2-s)+\frac12-\frac1p>0$.
The  local existence time in $M^{2,p}$ given by Theorem \ref{mainTh1} can be estimated as
\begin{equation}\label{xeq5}
\begin{split}
T_u&=\dfrac{c_s}{\|u_0\|_{M^{2,p}}(1+\|u_0\|_{M^{2,p}})}\\
& \gtrsim \dfrac{c_s}{\|u_0\|_{H^2}(1+\|u_0\|_{H^2})}\\
 &\geq \dfrac{c_s}{ N^{\vartheta}(1+ N^{\vartheta})} \\
&\geq \dfrac{c_s}{ N^{2\vartheta}}=:t_0.
 \end{split}
\end{equation}

Notice that, it is here the restriction $p\geq 2$ is necessary. Now, for $p\geq 2$, we have  $$(\|v_0\|_{M^{s,p}}+\|u_0\|_{M^{s,p}})(1+\|v_0\|_{M^{s,p}}+\|u_0\|_{M^{s,p}})\lesssim \|\eta_0\|_{H^s}(1+\|\eta_0\|_{H^s})=:C_s.$$ 
Therefore,
\begin{equation}\label{uxeq5}
\begin{split}
T_v &=\dfrac{c_s}{(\|v_0\|_{M^{s,p}}+\|u_0\|_{M^{s,p}})(1+\|v_0\|_{M^{s,p}}+\|u_0\|_{M^{s,p}})} \\
&\gtrsim \dfrac{c_s}{\|\eta_0\|_{H^s}(1+\|\eta_0\|_{H^s})} \\
&\geq \dfrac{c_s}{C_s} \geq t_0.
\end{split}
\end{equation}

In the light of the estimates \eqref{xeq5} and \eqref{uxeq5}, one can guarantee that the both  solutions $u$ and $v$ are  defined in the same time interval $[0,t_0]$.

From \eqref{xeq4} and \eqref{xeq5}, one has
\begin{equation}\label{0xeq4}
 t_0 \lesssim \dfrac{1}{E(u_0)}.
\end{equation}

Observe that, from \eqref{xeq6} the local solution $v \in M^{s,p}$ is given by 
\begin{equation}\label{xeq6-1}
v(x,t) = S(t)v_0 + h(x,t).
\end{equation}
Since  $[0, t_0]$ with $t_0 \sim  N^{-2\vartheta}$ is the common existence interval of $u$ and $v$, in view of \eqref{xeq6-1}, we can write  the solution $\eta$ in the form
\begin{equation}\label{00xeq9}
\eta(t)=u(t)+ v(t) = u(t)+ S(t)v_0 +h(t), \quad t\in [0,t_0].
\end{equation}

From \eqref{00xeq9}, at the time $t=t_0$, one has
\begin{equation}\label{xeq9}
\eta(t_0)=u(t_0)+ S(t_0)v_0 +h(t_0)=:u_1+v_1,
\end{equation}
where 
\begin{equation}\label{decomp1}
u_1=u(t_0) +h(t_0) \quad \textrm{and} \quad v_1 =S(t_0)v_0.
\end{equation} 

Now, at the time $t=t_0$  we consider the new initial data $u_1$, $v_1$ and evolve them according to the IVPs \eqref{xeq1} and \eqref{xeq2} respectively, and continue iterating this process. In each step of iteration we consider the decomposition of the initial data as in \eqref{xeq9}-\eqref{decomp1}. Continuing in this way, we have $v_1, \dots, v_k = S(k t_0)v_0$ with $\|v_k\|_{M^{s,p}} = \|v_0\|_{M^{s,p}}$. If we can guarantee that $u_1, \dots, u_k$ also have the same growth properties as that of $u_0$ then it can be  ensured that the common existence time interval $[0,t_0]$ remains the same in each iteration and one can glue them to cover the whole time interval $[0,T]$. 

In what follows, we show that the growth each $u_k$ has the same growth property as that of $u_0$  using induction argument. To give an idea we will prove only the case $k = 1$. The general case follows with  a similar argument. In the light of \eqref{0xeq4}, the main idea in the proof is the use of the energy conservation law \eqref{energy}. Notice that 
\begin{equation}\label{xeq10}
\begin{split}
E(u_1)&=E(u(t_0) +h(t_0))\\
&=E(u(t_0))+ \big[E(u(t_0) +h(t_0))-E(u(t_0))\big]\\
 &=:E(u(t_0))+\mathcal{X}.
\end{split}
\end{equation}

We estimate the second term in \eqref{xeq10} as follows
\begin{equation}\label{xeq11}
\begin{split}
\mathcal{X}&=2 \int_{\R}u(t_0)h(t_0) dx+\int_{\R}h(t_0)^2dx+2\gamma_1 \int_{\R}u_x(t_0)h_x(t_0) dx\\
 & \quad+\gamma_1\int_{\R}h_x(t_0)^2dx+2 \delta_1\int_{\R}u_{xx}(t_0)h_{xx}(t_0) dx+\delta_1\int_{\R}h_{xx}(t_0)^2dx\\
&\leq  2\|u(t_0)\|_{L^2}\|h(t_0)\|_{L^2}+\|h(t_0)\|_{L^2}^2+ \gamma_1 (2 \|u_x(t_0)\|_{L^2}\|h_x(t_0)\|_{L^2}+\|h_x(t_0)\|_{L^2}^2)\\
&\quad+\delta_1(2\|u_{xx}(t_0)\|_{L^2}\|h_{xx}(t_0)\|_{L^2}+\|h_{xx}(t_0)\|_{L^2}^2).
\end{split}
\end{equation}

Now, using estimate \eqref{estim1} from Lemma \ref{lemah1}, one obtains from \eqref{xeq11} that
\begin{equation}\label{xeq12}
\begin{split}
\mathcal{X} &\lesssim   N^{\frac12-\frac1p}N^{s-3}+N^{2(s-3)}+ \gamma_1 (N^{\frac12-\frac1p}N^{s-3}+N^{2(s-3)})+\delta_1(N^{2-s}N^{\frac12-\frac1p}N^{s-3}+ N^{2(s-3)})\\
& \lesssim  N^{\frac12-\frac1p}N^{s-3}\\
& \lesssim N^{-\frac12}.
\end{split}
\end{equation}
Inserting \eqref{xeq12} in   \eqref{xeq10}, we get
\begin{equation}\label{xeq13}
\begin{split}
E(u_1) & \leq E(u(t_0))+cN^{-\frac12}.
\end{split}
\end{equation}

Considering fixed time interval of length $t_0$ in each step of iteration, to cover the whole time interval $[0, T]$ for any given $T\gg 1$, one needs to iterate the process $\frac{T}{t_0}$ times.  Recalling \eqref{xeq5}, one has
\begin{equation}\label{T-t}
\dfrac{T }{t_0}\sim TN^{2\vartheta}.
\end{equation}
Thus, from \eqref{T-t} and \eqref{xeq13} it is clear that $E(u_1)$ posses the desired growth condition to cover $[0, T]$, provided
\begin{equation}\label{N-c}
TN^{2\vartheta} N^{-\frac12}  \lesssim N^{2\vartheta}.
\end{equation}
Note that the estimate in \eqref{N-c} remains valid for $1< s < 2$ if one  chooses
$N = N(T) = T^{2}$ large enough.   This completes  the proof of Theorem \ref{Th-global} for $1\leq \frac32-\frac1p<s<2$ and $p\geq 2$.\\

 \noindent
 {\bf Case 2. $\boxed{s\geq 2}$:} To prove the global well-posedness result in this case, we will use the previous case to derive a uniform {\em a priori} estimate. Let $p \geq 2$,
  $s_0\in (1,2)$ be fixed such that $s_0> \frac32-\frac1p$,  $\eta_0\in M^{s, p}$ with  $s \geq 2$ and $T\gg 1$ arbitrary. Using immersion, we have $\eta_0\in M^{s_0, p}$, and by {\bf Case 1}, $\eta \in C( [0, T],  M^{s_0, p})$.

From \eqref{Norm2} in Theorem A, we have $\|\eta(t)-S(t)\eta_0\|_{H^2}\lesssim (1+t)^{2-s}$. Therefore, using immersion, for all $t\in \R$, we have
\begin{equation}\label{estunifeta}
\begin{split}
\|\eta(t)\|_{M^{s_0, p}} &\leq \|\eta(t)-S(t)\eta_0\|_{M^{s_0, p}}+\|S(t)\eta_0\|_{M^{s_0, p}}\\
&\lesssim \|\eta(t)-S(t)\eta_0\|_{H^{s_0}}+\|\eta_0\|_{M^{s_0, p}}\\
&\lesssim (1+t)^{2-s_0}+\|\eta_0\|_{M^{s_0, p}}.
\end{split}
\end{equation}
Considering the local  existence time in \eqref{r2.45}, we can choose the time of existence $\delta$ of the solution $\eta$ as
$$
\delta=\dfrac{c}{\|\eta_0\|_{M^{s_0, p}}+\|\eta_0\|_{M^{s_0, p}}^2+T^{2(2-s_0)}}.
$$

For a fixed $0<t_0<T$, the Duhamel's formula with initial data $\eta(t_0)$ gives
\begin{equation}\label{D-M}
\eta(t_0+\tau)=S(\tau)\eta(t_0)-i\int_{0}^{\tau}S(t-t')F(\eta)(t_0+t') dt',\quad 0\leq\tau\leq \delta,
\end{equation}
where $F(\eta)= \tau (\partial_x)\eta^2 - \frac18\psi(\partial_x)\eta^3  -\frac7{48}\psi(\partial_x)\eta_x^2\,$.

Applying the $\|\cdot \|_{M^{s, p}}$ norm, using Lemma \ref{estimnormgwp} and the estimate \eqref{estunifeta},  for $t''=t_0+t'$ with $t'\in (0, \delta')$, we obtain from \eqref{D-M} that
\begin{equation}\label{Duhamelestim}
\begin{split}
\|\eta(t_0+\tau)\|_{M^{s, p}}&\leq \|\eta(t_0)\|_{M^{s, p}}+c\delta'\left( \|\eta(t'')\|_{M^{s_0, p}}+\|\eta(t'')\|_{M^{s_0, p}}^2 \right)\|\eta(t'')\|_{M^{s, p}}\\
&\leq \|\eta(t_0)\|_{M^{s, p}}+c\delta'\left(\|\eta_0\|_{M^{s_0, p}}+\|\eta_0\|_{M^{s_0, p}}^2+T^{2(2-s_0)}\right)\|\eta(t'')\|_{M^{s, p}}.
\end{split}
\end{equation}
In the same way as in the proof of the local well-posedness in $t_0=0$, we obtain the local existence time 
\begin{equation}\label{timedeltagwp}
\delta'=\dfrac{c}{\|\eta_0\|_{M^{s_0, p}}+\|\eta_0\|_{M^{s_0, p}}^2+T^{2(2-s_0)}}=\delta. 
\end{equation} 
Thus, considering the initial data at time $t=t_0$, we have the local well-posedness result in $M^{s, p}$ in the interval  $[t_0, t_0+\delta]$ with existence time length the same $\delta$ as was initially considering initial data at $t=0$. Therefore, we can iterate this process  finite number of times with a uniform existence time $\delta$ at each step. In fact, we will  iterate this to cover the  whole interval $[0,T]$ for any finite $T\gg 1$ given.  Note that one needs to iterate $\frac{T}\delta$ times to cover the whole interval $[0, T]$. Hence, we need to ensure that the constant accumulating in each step of iteration does not blow-up.

Considering $t_0=0$, choosing a convenient constant $c$ in \eqref{timedeltagwp}, and using a continuity argument in \eqref{Duhamelestim},  for any $t\in [0, \delta]$ one has
$$
\|\eta(t)\|_{M^{s, p}} \leq 2 \|\eta_0\|_{M^{s, p}}. 
$$

Similarly, if $t_0=\delta$, we have  for any $t\in [\delta, 2\delta]$
$$
\|\eta(t)\|_{M^{s, p}} \leq 2 \|\eta(\delta)\|_{M^{s, p}}\leq 2^2\|\eta_0\|_{M^{s, p}}.
$$
Proceeding in this way, after $\frac{T}{\delta}$ steps, for any $t\in [0, T]$, one has
\begin{equation*}
\begin{split}
\|\eta(t)\|_{M^{s, p}} &\leq 2^{\frac{T}{\delta}}\|\eta_0\|_{M^{s, p}}\\
&\leq C^{T \left(\|\eta_0\|_{M^{s_0, p}}+\|\eta_0\|_{M^{s_0, p}}^2+T^{2(2-s_0)}\right)}\|\eta_0\|_{M^{s, p}}.,
\end{split}
\end{equation*}
thereby proving a global well-posedness result as required.
\end{proof}


\setcounter{equation}{0}
\section{Local well-posedness in $H^{s,p}$ spaces}\label{sec-5}
In this section we will derive some multiplier estimates and prove the local well-posedness result for the IVP \eqref{eq1.9} with initial data in the $H^{s,p}$ spaces stated in Theorem \ref{mainTh2}. We start by introducing Fourier multiplier operators and their properties. 

Let $m$ be a measurable function defined on $\R^n$. Let $m(D)$ be an operator defined via the Fourier transform
$$
\widehat{m(D)f}(\xi):=m(\xi)\widehat{f}(\xi), \qquad \forall \; f\in \mathcal{S}(\R^n).
$$
The symbol $m(\xi)$ of the operator $m(D)$ is called a Fourier multiplier on $L^p(\R^n)$, $1\leq p\leq \infty$, if $m(D)$ can be extended as a bounded operator on $L^p(\R^n)$, i.e.,
$$m(D):L^p(\R^n)\to L^p(\R^n)$$
is bounded. We use
$\mathcal{M}_p(\R^n)$ to denote set of all Fourier multipliers on $L^p(\R^n)$. 

We record the following result from \cite{MW-16} that helps to verify if a bounded function is a  Fourier multiplier on $L^p(\R^n)$.

\begin{theorem}\label{teor2.1}
Let $k \in \N$, $k>\frac{n}{2}+1$, $m \in C^k(\R^n \setminus \{0\})$, and there exists $\delta\geq 0$ such that for any $|\alpha|\geq k$,
$$
|D^{\alpha} m(\xi)| \leq \dfrac{C_{\alpha}}{ \langle \xi \rangle^{\delta +\alpha}}, \qquad \forall\; \xi \in \R^n. 
$$
If $\delta=0$, then $m \in \mathcal{M}_p(\R^n),\; 1<p<\infty$.\\
If $\delta>0$, then $m \in \mathcal{M}_1(\R^n)$.
\end{theorem}

The following Sobolev embedding theorem will also be useful in our analysis.
\begin{theorem}\label{teor2.2} {\bf (a)} Let $0<sp<n$. Then
\begin{equation}\label{eq7}
\|f\|_{L^q} \lesssim\|f\|_{H^{s,p}}
\end{equation}
holds for $q\in [p,\frac{np}{n-sp}]$ if $1<p< \infty$, and  $q\in [p,\frac{np}{n-sp})$ if $p=1$.

{\bf (b)} Let $sp=n$, $1\leq p<\infty$. Then \eqref{eq7} holds for all $q\in [p, \infty)$.
\end{theorem}

\begin{proof}
For a detailed proof we refer to \cite{MW-16}.
\end{proof}

We will also use the fractional Leibniz rule stated below. 

\begin{proposition} 
If $s\geq 0$, $1\leq p<\infty$, the for all $f, g \in \s(\R^n)$
\begin{equation}\label{eq8}
\|\Lambda^s(fg)\|_{L^p} \lesssim\|f\|_{L^{p_1}}\|g\|_{H^{s,q_1}}+\|f\|_{H^{s,p_2}}\|g\|_{L^{q_2}}
\end{equation}
with $p_1, p_2, q_1,q_2 \in (1, \infty)$ satisfying $\frac1{p}=\frac1{p_1}+ \frac1{q_1}=\frac1{p_2}+ \frac1{q_2}$.
\end{proposition}
\begin{proof}
The proof can be found in \cite{KPV-93} for $1<p<\infty$ and in \cite{Gra-14} for $p=1$.
\end{proof}


We state the following result from \cite{MW-16} which is obtained while analysing the well-posedness issue for the BBM equation in the $H^{s,p}$ spaces.
\begin{proposition}\label{propomega}
If $p\in [1,\infty)$ and $s\geq \max\{\frac1{p}-\frac12, 0\}$, then 
\begin{equation*}
\|\omega(\partial_x)(u_1 u_2)\|_{H^{s,p}(\R)}\lesssim \|u_1 \|_{H^{s,p}(\R)}\|u_2\|_{H^{s,p}(\R)}.
\end{equation*}
\end{proposition}

Now, we move on to derive multilinear estimates to deal with the well-posedness of the IVP \eqref{eq1.9} with given data in  $H^{s,p}$ spaces.

\begin{proposition}\label{proptau}
If $p\in [1,\infty)$ and $s\geq \max\{\frac1{p}-\frac12, 0\}$, then 
\begin{equation}\label{eq4.4m}
\|\tau(\partial_x)(u_1 u_2)\|_{H^{s,p}(\R)}\lesssim \|u_1 \|_{H^{s,p}(\R)}\|u_2\|_{H^{s,p}(\R)}.
\end{equation}
\end{proposition}
\begin{proof}
First, note that the symbol $\dfrac{\tau(\xi)}{w(\xi)}$ of the  operator $\dfrac{\tau(D_x)}{\omega(D_x)}$  is in $\mathcal{M}_q$, $q\in [1, \infty]$. With this observation, one has
\begin{equation*}
\|\tau(\partial_x) h\|_{L^p}=\Big\|\dfrac{\tau(D_x)}{\omega(D_x)} \omega(D_x)h\Big\|_{L^p}\lesssim \|\omega(D_x)h\|_{L^p}.
\end{equation*}

Now, considering $h=\Lambda^s(u_1 u_2)$ and  applying the Proposition \ref{propomega} we conclude the proof.
\end{proof}

\begin{proposition}\label{propPsi}
If $p\in [1,\infty)$ and $s\geq \max\left\{ \frac1{p}+\frac12, 1 \right\}$, then 
\begin{equation}\label{eq4.6m}
\|\psi(\partial_x)(u_x v_x)\|_{H^{s,p}(\R)}\lesssim \|u \|_{H^{s,p}(\R)}\|v\|_{H^{s,p}(\R)}.
\end{equation}
\end{proposition}
\begin{proof}
To prove this lemma  we define Multiplier operators $T_1$
\begin{equation}\label{Prox1}
T_1:=\dfrac{\psi(\partial_x) \Lambda}{\omega(\partial_x)},
\end{equation}
with symbol $m_1(\xi)=\dfrac{\psi(\xi) (1+|\xi|^2)^{1/2}}{\omega(\xi)}$,
and $T_2$
\begin{equation}\label{Prox2}
T_2:=\dfrac{\Lambda^{s-1}\partial_x}{\Lambda^{s}},
\end{equation}
with symbol $m_2(\xi)=\dfrac{ (1+|\xi|^2)^{(s-1)/2} i\xi}{ (1+|\xi|^2)^{s/2}}$.

Note that $m_1$ and $m_2$ are Fourier multipliers on $L^q(\R)$, $1\leq q\leq \infty $, i.e., $m_1, m_2\in \mathcal{M}_q, \quad q\in [1, \infty]$.

Using \eqref{Prox1}, Proposition \ref{propomega} and \eqref{Prox2}, we have
\begin{equation*}
\begin{split}
\|\psi(\partial_x)(u_x v_x)\|_{H^{s,p}(\R)}&=\|\psi(\partial_x) \Lambda^s(u_x v_x)\|_{L^{p}}\\
&=\|\left(T_1 \dfrac{\omega(\partial_x)}{\Lambda}\right)\Lambda^s(u_x v_x)\|_{L^{p}}\\
&\lesssim\|\omega(\partial_x)\Lambda^{s-1}(u_x v_x)\|_{L^{p}}\\
&\lesssim\|\omega(\partial_x)(u_x v_x)\|_{H^{s-1,p}}\\
&\lesssim\|u_x\|_{H^{s-1,p}} \|v_x\|_{H^{s-1,p}}\\
&\lesssim\|u\|_{H^{s,p}} \|v\|_{H^{s,p}},
\end{split}
\end{equation*}
where we used the Proposition \ref{propomega} with $s-1\geq \max\left\{0, \frac1{p}-\frac12 \right\}$.
\end{proof}

\begin{proposition}\label{propPsi-1}
If $p\in [1,\infty)$ and $s\geq 1$, 
then 
\begin{equation}\label{eq4.10m}
\|\psi(\partial_x)(uvw)\|_{H^{s,p}(\R)}\lesssim \|u \|_{H^{s,p}(\R)}\|v\|_{H^{s,p}(\R)}\|w\|_{H^{s,p}(\R)}.
\end{equation}
\end{proposition}
\begin{proof}
Without loss of generality, we can assume  that $u=v=w=\eta$, a similar proof works in the general case. Since the symbol $\frac{\psi(\xi)}{i\xi(1+|\xi|^2)^{-2}}$ of the operator  $\frac{\psi(\partial_x)}{\partial_x \Lambda^{-4}}$ is in $\mathcal{M}_q, \; q\in [1, \infty], $
to get \eqref{eq4.10m}, it is enough to prove 
\begin{equation}\label{estimp}
\|\partial_x \Lambda^{-4}(\eta^3)\|_{H^{s,p}(\R)}=\|\left(\partial_x \Lambda^{\frac2{p} -4}\right)\left(\Lambda^{-\frac2{p} }\right)(\eta^3)\|_{H^{s,p}(\R)} \lesssim \|\eta\|_{H^{s,p}(\R)}^3.
\end{equation}

 We divide the proof in two different cases.\\

\noindent
{\bf Case I. $\boxed{p\geq 3}$:} 
Note that for all $p\geq 3$ the symbol of the operator
$\partial_x \Lambda^{\frac2{p} -4}$ is in $ \mathcal{M}_q$, for  $q\in [1, \infty]$.
Also, for all $p\geq 3$, one has $p\in [\frac{p}3,\frac{\frac{p}3}{1-\frac23 }]$ and consequently using Theorem \ref{teor2.2} we have $\|f\|_{L^p}\leq \|f\|_{H^{\frac2{p}, \frac{p}3}}$. Therefore,
\begin{equation}\label{op2}
\Lambda^{-\frac2{p} }:  \,\, L^{\frac{p}3} \to  L^{p}.
\end{equation}

 Now,  using that  the symbol of the operator
$\partial_x \Lambda^{\frac2{p} -4}$ is in $ \mathcal{M}_q$, for  $q\in [1, \infty]$, \eqref{op2} and the fractional Leibniz rule \eqref{eq8}, we obtain
\begin{equation*}
\begin{split}
\|\left(\partial_x \Lambda^{\frac2{p} -4}\right)\left(\Lambda^{-\frac2{p} }\right)(\eta^3)\|_{H^{s,p}(\R)} &\lesssim \|\Lambda^{-\frac2{p} }\Lambda^{s }(\eta^3) \|_{L^p}\\
&\lesssim \|\Lambda^{s }(\eta^3) \|_{L^{\frac{p}3}}\\
&\lesssim \|\Lambda^{s }(\eta) \|_{L^{p}} \|\eta\|_{L^{p}}^2\\
&\lesssim \|\Lambda^{s }(\eta) \|_{L^{p}}^3.
\end{split}
\end{equation*}
where in the last inequality we used that $\Lambda^{-s }$ is an operator with symbol in $\mathcal{M}_p$.\\

\noindent
{\bf Case II.  $\boxed{1\leq p< 3}$:} 
In order to prove \eqref{estimp}, first note that, using Theorem \ref{teor2.1} one can easily infer that the operator $\partial_x \Lambda^{-3}: L^p\to L^p$ is bounded. Using this information, one has
\begin{equation}\label{estimp1}
\begin{split}
\|\partial_x \Lambda^{-4}(\eta^3)\|_{H^{s,p}(\R)}&=\|\left(\partial_x \Lambda^{-3}\right)\left(\Lambda^{\frac1{p}-1- }\right)\left(\Lambda^{s-\frac1{p} +}\right)(\eta^3) \|_{L^p}\\
&\lesssim\|\left(\Lambda^{\frac1{p}-1- }\right)\left(\Lambda^{s-\frac1{p} +}\right)(\eta^3) \|_{L^p}.
\end{split}
\end{equation}

Furthermore, for all $1\leq p< 3$, one has $\Lambda^{\frac1{p}-1- }: L^p \to L^1$. Indeed the case $p=1$ is obvious by Theorem \ref{teor2.1} and in the case $1<p<3$ using  Theorem \ref{teor2.2}, one obtains
\begin{equation}\label{estimp2}
\|f\|_{L^p}\lesssim \|f\|_{H^{s_p,1}} \Longleftrightarrow \|\Lambda^{\frac1{p}-1- }f\|_{L^p}\lesssim \|f\|_{L^1}, 
\end{equation}
where $s_p=-\frac1{p}+1+\epsilon$, $0<\epsilon\ll1$,  since $p\in [1, \frac{1}{1-s_p}]=[1, \frac{p}{1-\epsilon p}]$.

Now, combining \eqref{estimp1} and \eqref{estimp2},  and aplying the fractional Leibniz rule \eqref{eq8} twice, we get
\begin{equation*}
\begin{split}
\|\partial_x \Lambda^{-4}(\eta^3)\|_{H^{s,p}(\R)}
&\lesssim\|\Lambda^{s-\frac1{p} +}(\eta^3) \|_{L^1}\\
&\lesssim\|\Lambda^{s-\frac1{p} +}(\eta^2) \|_{L^{\frac32}}\|\eta\|_{L^{3}}+\|\Lambda^{s-\frac1{p} +}(\eta) \|_{L^{3}}\|\eta^2\|_{L^{\frac32}} \\
&\lesssim  \|\Lambda^{s-\frac1{p} +} \eta \|_{L^{3}}\|\eta\|_{L^{3}}^2. \\
\end{split}
\end{equation*}

In what follows, we will show that 
\begin{equation*}
 \|\Lambda^{s-\frac1{p} +}\eta \|_{L^{3}} \lesssim \|\Lambda^{s}\eta \|_{L^{p}}.
\end{equation*}
This inequality is equivalent to
\begin{equation}\label{estimp6}
  \|\eta \|_{L^{3}} \lesssim \|\Lambda^{S_p}\eta \|_{L^{p}},
\end{equation}
where $S_p=\frac1{p}-\epsilon$, $0<\epsilon \ll1$, which is true by Theorem \ref{teor2.2}, since $3\in [p, \frac{p}{1-pS_p})=[p, \frac1{\epsilon})$, for $1\leq p <3$.

Finally, to finish the proof we need to show that
\begin{equation*}
 \|\eta \|_{L^{3}} \lesssim \|\Lambda^{s}\eta \|_{L^{p}},
\end{equation*}
which is a consequence of \eqref{estimp6} and the fact that $\frac{\Lambda^{S_p}}{\Lambda^{s}}$ is a multiplier operator with symbol  in $\mathcal{M}_p$.
\end{proof}

In what follows we find an estimate for the unitary group operator $S(t)=e^{-i\phi(D)t}$  in the $H^{s,p}$ spaces. This estimate plays fundamental role while implementing contraction mapping principle.

\begin{remark}\label{rem-41}
Let $\mathcal{L}(H)$ be the set of the bounded  linear operators defined on a space $H$. If    $T$ is a multiplier operator in $L^p$, $ 1\leq p\leq \infty$, i.e.,  symbol of $T$ is Fourier multiplier in $ \mathcal{M}_p$ and $T$ commutes with $\Lambda^{s}$, then $T\in \mathcal{L}(H^{s,p})$.  In fact
$$
\|Tf\|_{H^{s,p}}=\|\Lambda^{s}T f\|_{L^p}=\|T\Lambda^{s} f\|_{L^p}\lesssim\|\Lambda^{s} f\|_{L^p}=\|f\|_{H^{s,p}}
$$
and
\begin{equation*}
\|T\|_{\mathcal{L}(L^p)}=\|T\|_{\mathcal{L}(H^{s,p})}
\end{equation*}
\end{remark}

\begin{lemma}\label{lema-ln-m}
Let $s\in \R$, $1\leq p\leq\infty$ and $S(t)$  be the unitary operator defined in \eqref{St}. Then for $t>0$, one has
\begin{equation}\label{lin-10}
\|S(t)\|_{\mathcal{L}(H^{s,p})}\lesssim \langle t\rangle^{2\left|\frac12-\frac1p\right|}.
\end{equation}

\end{lemma}

\begin{proof}
Let $\delta\ne 0$. By Plancherel identity it is clear that 
\begin{equation}\label{liml2}
S(t) \in \mathcal{L}(L^2) \quad \textrm{and} \quad \|S(t)\|_{\mathcal{L}(L^2)}=1. 
\end{equation}

Let $\delta_2=0$. Then, applying the Bernstein's Theorem (see Lemma 2.1 in \cite{MH} and \cite{MW-16}), we obtain
\begin{equation}\label{eq4.21}
\begin{split}
\|(S(t)-I)f\|_{\mathcal{L}(L^1)}\lesssim \|e^{-i\phi(\xi)t}-1)\|_{L^2}^{\frac12}\,
\|\partial_{\xi}(e^{-i\phi(\xi)t}-1)\|_{L^2}^{\frac12}.
\end{split}
\end{equation}
Clearly
\begin{equation}\label{eq4.24}
|e^{-i\phi(\xi)t}-1|=\left|\int_0^t \dfrac{d }{d \tau}\left( e^{-i\phi(\xi)\tau} \right) d\tau\right|=\left|\phi(\xi)\,\int_0^t e^{-i\phi(\xi)\tau}\right|\lesssim \dfrac{|t|}{\langle \xi \rangle},
\end{equation}
holds for $\delta=0$.
Also, we have that
\begin{equation}\label{eq4.25}
  |\partial_{\xi}(e^{-i\phi(\xi)t}-1)|\lesssim \dfrac{|t|}{\langle \xi \rangle^2}.
\end{equation}
Consequently, in view of \eqref{eq4.24} and \eqref{eq4.25}, the estimate \eqref{eq4.21}  yields
\begin{equation}\label{est-ll2}
\|(S(t)-I)f\|_{\mathcal{L}(L^1)}\lesssim |t|.
\end{equation}
 
 From the estimate \eqref{est-ll2} one can infer that
 \begin{equation}\label{liml1}
\|S(t)\|_{\mathcal{L}(L^1)}\lesssim \langle t\rangle.
\end{equation}
By interpolation between \eqref{liml2} and \eqref{liml1} and using duality, we get
$$
\|S(t)\|_{\mathcal{L}(L^p)}\lesssim \langle t\rangle^{2\left|\frac12-\frac1p\right|}.
$$

Finally, using Remark \ref{rem-41}, we conclude that
$$
\|S(t)\|_{\mathcal{L}(H^{s,p})}\lesssim \langle t\rangle^{2\left|\frac12-\frac1p\right|},
$$
as required.
\end{proof}

Now, having the linear and  multilinear estimates at hand, we are in position to provide a proof for the second main result of this work stated in Theorem \ref{mainTh2}.

\begin{proof}[Proof of Theorem \ref{mainTh2}]
As in the proof of Theorem \ref{mainTh1}, here too, we consider the application $\Psi$ defined in \eqref{eq3.42} and use estimates \eqref{lin-10}, \eqref{eq4.4m}, \eqref{eq4.6m} and \eqref{eq4.10m} respectively from Lemma \ref{lema-ln-m}, Propositions \ref{proptau}, \ref{propPsi} and \ref{propPsi-1}, and prove that it is a contraction map on a ball in $H^{s,p}(\R)$ space. Rest of the proof follows with a standard argument. So we omit the details.
\end{proof}

\begin{remark}

For $\delta\ne 0$, the estimate \eqref{eq4.24} doesnot hold and consequently we could not prove the local well-posedness result fot eh IVP \eqref{5kdvbbm} in $H^{s,p}$, $p\ne 2$. So, this is an interesting open problem.

Once having obtained the local well-posedness results for given data in the modulation spaces $M^{s,p}$ and the $L^p$-based Sobolev spaces $H^{s,p}$, two questions arise naturally. The first one is, whether the regularity requirement in the initial data is optimal. The second question is, whether one can extend the local solution globally in time. In Section \ref{sec-4} we partially  responded  the second question  by proving the global well-posedness result in $M^{s,p}$ spaces.  We are working on other questions in our ongoing project and will be made public soon.
\end{remark}

\section{Solitary Wave solution}\label{sec-6}

In this section we will derive an explicit  solitary wave solution of \eqref{5kdvbbm}. First, recall that the IVP \eqref{5kdvbbm} possesses hamiltonian structure when $\gamma= \frac7{48}$. This restriction on $\gamma$  implies $\gamma_1=\gamma_2=\frac1{12}$.  

 For $c\in \R$, we look for the solution of the form $\eta(x,t)=\phi (x-ct)$ where $\phi \in \mathbb{S}(\R)$. Note that, if $\eta(x,t)=\phi (x-ct)$ is a solution of \eqref{5kdvbbm}, then we have
\begin{equation}\label{xeqx1}
(1-c)\phi'+\frac1{12}(c+1)\phi'''+(\delta_2-\delta_1c)\phi'''''+ (\frac34 \phi^2)'-\frac18 (\phi^3)'+\frac{7}{48}(\phi^2)'''-\frac{7}{48}(\,(\phi')^2)'=0.
\end{equation}

Taking $c=\frac{\delta_2}{\delta_1}$ and multiplying \eqref{xeqx1} by $48$, we get
\begin{equation}\label{eqx1}
48(1-c)\phi+4(c+1)\phi''+ 36 \phi^2-6 \phi^3+14\phi''\phi+7(\phi')^2=0.
\end{equation}

 Considering a solution of \eqref{eqx1} in the form 
 \begin{equation}\label{xeqx2}
 \phi(x)=A \operatorname{sech}^2 (Bx),
 \end{equation}
  we have
 \begin{equation}\label{xeqx23}
 \phi'(x)=-2AB \operatorname{sech}^2 (Bx)\tanh(Bx), \quad  \phi''(x)=4AB^2\operatorname{sech}^2 (Bx)-6AB^2\operatorname{sech}^4 (Bx)).
 \end{equation}
 
 Inserting \eqref{xeqx2} and \eqref{xeqx23} in \eqref{eqx1}, we obtain
\begin{equation}\label{Sol-5}
\begin{split}
16A \operatorname{sech}^2(Bx)&\left[ (1+c) B^2+3(1-c)\right]+12A\operatorname{sech}^4(Bx)\left[ -2B^2(1+c)+7AB^2+3A \right]\\
&-2A^2\operatorname{sech}^6(Bx)\left[ 56B^2+3A \right]=0.
\end{split}
\end{equation}
 
 The relation \eqref{Sol-5} is satisfied if
 \begin{equation}\label{Sol-6}
  \quad B^2=\frac{3(c-1)}{c+1}, \quad -2B^2(1+c)+7AB^2+3A=0, \quad A=-\frac{56}{3}B^2.
 \end{equation}

The system of equations in \eqref{Sol-6}, yields
 \begin{equation}\label{sol-c}
 c=-14\sqrt{66}-113=\frac{\delta_2}{\delta_1},  \quad A=-\frac{28}{7}(6+\sqrt{66}), \quad B^2=\frac{3}{14}(6+\sqrt{66}).
 \end{equation}
 
 Finally, using the constants from \eqref{sol-c} in \eqref{xeqx2} we obtain the following  explicit form of the solitary wave solution to the higher order water wave model \eqref{5kdvbbm} 
\begin{equation*}
\begin{split}
\eta(x,t)&=\phi\left(x+(14\sqrt{66}+113)t\right)\\
&= -\frac{28}{7}(6+\sqrt{66})\operatorname{sech}^2\left(\sqrt{\frac{3}{14}(6+\sqrt{66})}\,\left(x+(14\sqrt{66}+113)t\right)\right).
\end{split}
\end{equation*}

\begin{remark}
In our ongoing project, we are working on the stability issues of the solitary wave solution presented above. Also, we plan to use this explicit form to explore the ill-posedness results in the modulation spaces and $L^p$-based Sobolev spaces.
\end{remark}


\section*{Acknowledgments}
The first author is thankful to the Department of Mathematics, IMECC-UNICAMP for pleasant hospitality where part of this work was developed. The second author acknowledges  grants from FAPESP, Brazil (\# 2023/06416-6) and CNPq, Brazil (\# 307790/2020-7). 


\end{document}